\documentclass[12pt]{amsart}
\usepackage{amssymb}
\usepackage{amsthm}
\usepackage{adjustbox,dsfont,epsfig}
\usepackage[all]{xy}
\usepackage[colorlinks=true,linkcolor=magenta,citecolor=magenta]{hyperref}

\newtheorem{theorem}{Theorem}[section]
\newtheorem{definition}[theorem]{Definition}
\newtheorem{proposition}[theorem]{Proposition}
\newtheorem{corollary}[theorem]{Corollary}

\newtheorem{lemma}[theorem]{Lemma}
\theoremstyle{remark}
\newtheorem{example}[theorem]{Example}
\theoremstyle{remark}
\newtheorem{remark}[theorem]{Remark}

\begin{document}

\title[Idempotent states on quantum permutation groups]{Analysis for idempotent states on quantum permutation groups}

\author{J.P. McCarthy}
\address{\parbox{\linewidth}{Department of Mathematics, Munster Technological University, Cork, Ireland.\\ \texttt{jp.mccarthy@mtu.ie}}}


\begin{abstract}
Woronowicz proved the existence of the Haar state for compact quantum groups under a separability assumption later removed by Van Daele in a new existence proof. A minor adaptation of Van Daele's proof yields an idempotent state in any non-empty weak-* compact convolution-closed convex subset of the state space. Such subsets, and their associated idempotent states, are studied in the case of quantum permutation groups.
\end{abstract}
\setcounter{tocdepth}{1}
\maketitle

\baselineskip=16.5pt
\tableofcontents
\baselineskip=14pt

\section*{Introduction}
It is sometimes quipped that \emph{quantum groups are neither quantum nor groups}. Whatever about compact quantum groups not being quantum, compact quantum groups are, of course, not in general classical groups. On the other hand, compact Hausdorff groups \emph{are} compact quantum groups. Furthermore, the classical theorems of the existence of the Haar measure, Peter--Weyl, Tannaka--Krein duality, etc., can all be viewed as special cases of the quantum analogues proved by Woronowicz \cite{wo2,wo3}, and thus naturally the theory of compact quantum groups has many commonalities with the theory of compact groups.

\bigskip

Not all classical theorems generalise so nicely:
\begin{theorem}[Kawada--It\^{o} Theorem, (Th. 3, \cite{kaw})]\label{kawito}
Let $G$ be a compact separable group. Then a probability distribution on $G$ is idempotent with respect to convolution if and only if it is the uniform distribution on a closed subgroup $H\leq G$.
\end{theorem}
The quantum analogue of a closed subgroup, $\mathbb{H}\leq \mathbb{G}$, is given by a comultiplication-respecting surjective $\ast$-homomorphism $\pi:C(\mathbb{G})\to C(\mathbb{H})$, and  the direct quantum analogue of the Kawada--It\^{o} theorem would be that each state idempotent with respect to convolution is a \emph{Haar idempotent}, that is a state on $C(\mathbb{G})$ of the form $h_{C(\mathbb{H})}\circ \pi$ (where  $h_{C(\mathbb{H})}$  is the Haar state on $C(\mathbb{H})$). However in 1996 Pal discovered non-Haar idempotents in the Kac--Paljutkin quantum group \cite{pal}, and thus the direct quantum analogue of the Kawada--It\^{o} theorem is false (in fact there are counterexamples in the dual of $S_3$, an even `smaller' quantum group \cite{frs}).

\bigskip

The null-spaces of Pal's idempotent states are only left ideals. Starting with \cite{frs},  Franz, Skalski and coauthors  undertook a general and comprehensive study of idempotent states on compact quantum groups, and, amongst other results, showed that the null-space being a one-sided rather than two-sided ideal is the only obstruction to an idempotent being Haar (Proposition \ref{fst3.3}). This study continued into the locally compact setting: a history of this whole programme of study, with references, is summarised in the introduction of Kasprzak and So{\l}tan \cite{kap}. The current work is supplemental to this study: giving a new way of viewing idempotent states, a new way of thinking about non-Haar idempotents (with group-like support projection), giving new examples of idempotent states (see Section \ref{stab}), and concentrating on idempotent states related to the inclusion of the classical permutation group in the quantum permutation group. The latter of these explains the (non-standard) use of the bidual as a von Neumann algebra in this work. The conventional choice of von Neumann algebra in the study of compact quantum groups is the algebra:
$$L^\infty(\mathbb{G}):=C_{\text{r}}(\mathbb{G})'',$$
but as the reduced algebra $C_{\text{r}}(\mathbb{G})$ does not in general admit a character, it cannot `see' the classical version $G\leq \mathbb{G}$, so instead  the bidual $C(\mathbb{G})^{**}$ is used.

\bigskip

In the case of quantum permutation groups, interpreting elements of the state space as   quantum permutations, called the Gelfand--Birkhoff picture in \cite{mcc}, leads to the consideration of distinguished \emph{subsets} of the state space. In \cite{mcc}, using the fact that idempotent states in the case of finite quantum groups have group-like support (Cor. 4.2, \cite{frs}), \emph{subsets} of the state space are associated to idempotent states. The current work generalises this point of view: the subset associated to an idempotent state $\phi$  on a  quantum permutation group $\mathbb{G}$ is called a \emph{quasi}-subgroup (after \cite{kap}), and given by the set of states absorbed by the idempotent:
$$\mathbb{S}_{\phi}=\{\varphi\in\mathcal{S}(C(\mathbb{G}))\colon\,\varphi\star \phi=\phi=\phi\star \varphi\}.$$
Whenever a quasi-subgroup is given by a (universal) Haar idempotent, it is stable under \emph{wave-function collapse} (see Definition \ref{condition}). There is an obvious relationship between ideals and wave-function collapse: that  all classical quasi-subgroups correspond to subgroups  is just another way of saying that there are no one-sided ideals in the commutative case. An equivalence between Haar idempotent states and the stability of the associated quasi-subgroup under wave-function collapse is not proven here, but there is a partial result  (Theorem \ref{wfct}, for which a key lemma is Remark \ref{Vaes} of Stefaan Vaes).

\bigskip

The other theme of the study of Franz, Skalski and coauthors is the relationship between idempotent states and group-like projections, and culminates in a comprehensive statement about idempotent states being group-like projections in the multiplier algebra of the dual discrete quantum group \cite{frs}. This work contains no such comprehensive statement, but does extend the definition of continuous group-like projections $p\in C(\mathbb{G})$ to group-like projections $p\in C(\mathbb{G})^{**}$, the bidual.

\bigskip

Idempotent states with group-like support projection are particularly well-behaved, however it is shown that in the non-coamenable case the support projection of the Haar state is not group-like.

\bigskip

The consideration of subsets of the state space leads directly to the key observation in this work that non-empty weak-* compact convolution-closed convex subsets $\mathbb{S}$ of the state space, which are termed Pal sets, contain $\mathbb{S}$-invariant idempotent states $\phi_{\mathbb{S}}$:
$$\varphi\star \phi_{\mathbb{S}}=\phi_{\mathbb{S}}=\phi_{\mathbb{S}}\star \varphi\qquad (\varphi\in\mathbb{S}).$$
  This observation is via Van Daele's proof of the existence of the Haar state \cite{van} (ostensibly for the apparently esoteric and pathological non-separable case).   This observation yields new examples of (generally) non-Haar idempotent states in the case of quantum permutation groups: namely from the stabiliser quasi-subgroups of Section \ref{stab}. Pal sets, through their idempotent state, generate quasi-subgroups. Consider $S_3\leq S_4^+$ via $C(S_4^+)\to C(S_4^+)/\langle u_{44}=1\rangle$: this study yields the interesting example of an intermediate quasi-subgroup
  $$S_3\subsetneq (S_4^+)_4\subsetneq S_4^+.$$ Where $h$ is the Haar state on $C(S_4^+)$, the (non-Haar) idempotent in $(S_4^+)_4$ is given by:
   $$h_4(f)=\frac{h(u_{44}fu_{44})}{h(u_{44})}\qquad (f\in C(S_4^+)).$$
   This quasi-subgroup shares many properties of the state space of $C(S_3)$, namely it is closed under convolution, closed under reverses ((5.1), \cite{mcc}), and contains an identity for the convolution (i.e. the counit). Moreover, if any quantum permutation $\varphi\in (S_4^+)_4$ is measured with $u_{44}\in C(S_4^+)$ (in the sense of the Gelfand--Birkhoff picture), it gives one with probability one (i.e. it fixes label four). However, while it contains states non-zero on the commutator ideal of $C(S_4^+)$, this isn't a quantum permutation group on three labels because $(S_4^+)_4$ is not closed under wave-function collapse (the null-space of $h_4$ is one-sided).

\bigskip

A famous open problem in the theory of quantum permutation groups is the maximality conjecture: that the classical permutation group $S_N\leq S_N^+$ is a maximal quantum subgroup. This work undertakes some general analysis for the support projections of characters on algebras of continuous functions on quantum permutation groups.   Following on from Section 6.3 of \cite{mcc}, the current work considers the possibility of an \emph{exotic} intermediate quasi-subgroup strictly between the classical and quantum permutation groups. An attack on the maximality conjecture via such methods is not \emph{a priori} particularly promising, but some basic analysis of the support projections of the characters might be useful in the future.  This analysis shows that the support projection of the Haar idempotent $h_{S_N}$ associated with $S_N\leq S_N^+$ is a group-like projection in the bidual. One consequence of this is Theorem \ref{idemtotty} which says that  $h_{S_N}$ and  \emph{any} genuinely quantum permutation generates a quasi-subgroup strictly bigger than $S_N$, i.e. an idempotent state between  $h_{S_N}$ and the Haar state on $C(S_N^+)$. It isn't  $h_{S_N}$, but it could be (1) a non-Haar idempotent; or, for some $N\geq 6$, (2) the Haar idempotent from an exotic quantum subgroup  $S_N\leq \mathbb{G}_N\leq S_N^+$; or (3) the Haar state on $C(S_N^+)$. If it is always (3), a strictly stronger statement than the maximality conjecture, then the maximality conjecture holds.

\bigskip

Using the Gelfand--Birkhoff picture, this particular analysis allows us to consider the (classically) random and truly quantum parts of a quantum permutation, and there are some basic rules governing the convolution of (classically) random quantum permutations and truly quantum permutations. Some consequences of these are explored: for example, an idempotent state on $C(S_N^+)$ is either random, or ``less than half'' random (Corollary \ref{idempot2}).

\bigskip

The paper is organised as follows. Section 1 introduces  compact quantum groups, and discusses Van Daele's proof of the existence of the Haar state. Key in this work is the restriction to universal algebras of continuous functions (and the bidual von Neumann algebra), and the reasons for this restriction are explained. A further restriction to quantum permutation groups is made, and finally some elementary properties of the bidual are summarised. Section 2 introduces Pal sets, and asserts that they contain idempotent states. Quasi-subgroups are defined to fix the non-injectivity of the association of a Pal set to its idempotent state. The definition of a group-like projection is extended to group-like projections in the bidual, and the interplay between such group-like projections and idempotent states is explored. Wave-function collapse is defined, and the question of stability of a quasi-subgroup under wave-function collapse studied. In Section 3, stabiliser quasi-subgroups are defined, and it is shown that there is a strictly intermediate quasi-subgroup between $S_{N-1}^+\leq S_N^+$ and $S_N^+$. In Section 4, exotic quasi-subgroups of $S_N^+$ are considered (and by extension exotic quantum subgroups). Necessarily this section talks about the classical version of a quantum permutation group. The support projections of characters are studied, and it is proved that the sum of these is a group-like projection in the bidual. In the case of $S_N^+$, this group-like projection is used to define the (classically) random and truly quantum parts of a quantum permutation, and it is proven that the Haar idempotent coming from $S_N\leq S_N^+$ together with a quantum permutation with non-zero truly quantum part generates a non-classical quasi-subgroup in $S_N^+$ that is strictly bigger than $S_N$ (but possibly equal to $S_N^+$). In Section 5 the convolution of random and truly quantum permutations is considered, and as a corollary a number of quantitative and qualitative results around the random and truly quantum parts of convolutions. In Section 6 there is a brief study of the number of fixed points of a quantum permutation, and it is shown that as a corollary of never having an integer number of fixed points, the Haar state is truly quantum.

\section{Compact quantum groups}
\subsection{Definition and the Haar state}
\begin{definition}
\textbf{An} \emph{algebra of continuous functions on a ($\mathrm{C}^*$-algebraic) compact quantum group} $\mathbb{G}$ is a $\mathrm{C}^*$-algebra $C(\mathbb{G})$ with unit $\mathds{1}_{\mathbb{G}}$ together with a unital $\ast$-homomorphism  $\Delta:C(\mathbb{G})\to C(\mathbb{G})\otimes C(\mathbb{G})$ into the minimal tensor product that satisfies coassociativity and \emph{Baaj--Skandalis cancellation}:
 $$\overline{\Delta(C(\mathbb{G}))(\mathds{1}_{\mathbb{G}}\otimes C(\mathbb{G}))}=\overline{\Delta(C(\mathbb{G}))(C(\mathbb{G})\otimes \mathds{1}_{\mathbb{G}})}=C(\mathbb{G})\otimes C(\mathbb{G}).$$
\end{definition}
Woronowicz defined compact matrix quantum groups \cite{wo1}, and extended this definition to compact quantum groups \cite{wo3}. In order to establish the existence of a Haar state, Theorem \ref{THQIH} below, Woronowicz assumed that the algebra of functions was separable. Shortly afterwards Van Daele removed this condition \cite{van}, and established the existence of a Haar state in the non-separable case. The quantum groups in the current work are compact matrix quantum groups, which are separable, however Van Daele's proof will be teased out  and then adapted in Section \ref{palquasi}. Note that while Lemmas \ref{van2.1} and \ref{van2.2} are attributed here to Van Daele, it is pointed out by Van Daele that the techniques of their proofs  were largely present in the work of Woronowicz.

\bigskip

Define the convolution of states $\varphi_1,\,\varphi_2$ on $C(\mathbb{G})$:
$$\varphi_1\star \varphi_2:=(\varphi_1\otimes\varphi_2)\Delta.$$

\begin{theorem}[\cite{van,wo3}]\label{THQIH}
The algebra of continuous functions $C(\mathbb{G})$ on a compact quantum group admits a unique invariant state $h$, such that for all states $\varphi$ on $C(\mathbb{G})$:
$$h\star \varphi=h=\varphi\star h.$$
\end{theorem}
\begin{lemma}[Lemma 2.1, \cite{van}]\label{van2.1}
Let $\varphi$ be a state on $C(\mathbb{G})$. There exists a state $\phi_{\varphi}$ on $C(\mathbb{G})$ such that
$$\varphi\star\phi_{\varphi}=\phi_{\varphi}=\varphi\star\phi_{\varphi}.$$
\end{lemma}
\begin{proof}
  Define
  $$\varphi_n=\frac{1}{n}(\varphi+\varphi^{\star 2}+\cdots+\varphi^{\star n}).$$
  As the state space $\mathcal{S}(C(\mathbb{G}))$ is convex and closed under convolution, $(\varphi_n)_{n\geq 1}\subset\mathcal{S}(C(\mathbb{G}))$. Via the weak-* compactness of the state space, Van Daele shows that $\phi_{\varphi}$, a weak-* limit point of $(\varphi_n)_{n\geq 1}$, is $\varphi$-invariant.
\end{proof}
\begin{lemma}[Lemma 2.2, \cite{van}]\label{van2.2}
  Let $\varphi$ and $\phi$ be states on $C(\mathbb{G})$ such that $\varphi\star\phi=\phi$. If $\rho\in C(\mathbb{G})^*$ and $0\leq \rho\leq \varphi$, then also $\rho\star\phi=\rho(\mathds{1}_{\mathbb{G}})\phi$.
\end{lemma}

\begin{proof}[Proof of Theorem \ref{THQIH}]
  Where $\mathcal{S}(C(\mathbb{G}))$ is the state space of $C(\mathbb{G})$, for each positive linear functional $\omega$ on $C(\mathbb{G})$, define:
  $$K_{\omega}:=\{\varphi\in\mathcal{S}(C(\mathbb{G}))\,\colon\,\omega\star\varphi=\omega(\mathds{1}_{\mathbb{G}})\varphi\}.$$
  As per Van Daele, $K_\omega$ is closed and thus compact with respect to the weak-* topology. It is non-empty because $\omega$ can be normalised to a state $\widehat{\omega}$ on $C(\mathbb{G})$, and by Lemma \ref{van2.1}, there exists $\phi_\omega\in K_{\widehat{\omega}}$ and thus $\phi_\omega\in K_\omega$.

  \bigskip

  Let $\phi\in K_{\omega_1+\omega_2}$. Note that both $\omega_1,\omega_2\leq \omega_1+\omega_2$, and so by Lemma \ref{van2.2}, $\phi\in K_{\omega_1}\cap K_{\omega_2}$ so that:
  $$K_{\omega_1+\omega_2}\subseteq K_{\omega_1}\cap K_{\omega_2}.$$
  Assume that the intersection of the $K_\omega$ over the positive linear functionals on $C(\mathbb{G})$ is empty. Thus, where the complement is with respect to $\mathcal{S}(C(\mathbb{G}))$:
  $$\bigcup_{\omega\text{ pos. lin. func.}}K_\omega^c=\mathcal{S}(C(\mathbb{G})),$$
  is an open cover of a compact set, and thus admits a finite subcover $\{K_{\omega_i}^c\colon i=1,\dots,n\}$ such that
  $$\bigcup_{i=1}^n K_{\omega_i}^c=\mathcal{S}(C(\mathbb{G}))\implies \bigcap_{i=1}^nK_{\omega_i}=\emptyset.$$
  Let $\psi=\sum_{i=1}^n\omega_i$: the set $K_\psi$ is non-empty. It is also a subset of:
  $$ \bigcap_{i=1}^nK_{\omega_i}=\emptyset,$$
  an absurdity, and so the intersection of all the $K_\omega$ is non-empty, and thus there is a state $h$ that is left-invariant for all positive linear functionals and thus for $\mathcal{S}(C(\mathbb{G}))$.
\end{proof}
\subsection{The universal and reduced versions}
A reference for this section is Timmermann \cite{tim}. A compact quantum group $\mathbb{G}$ has a dense Hopf $\ast$-algebra of regular functions, $\mathcal{O}(\mathbb{G})$. The algebra of regular functions has a minimal $\mathrm{C}^*$-norm completion, the reduced algebra of continuous functions, $C_{\text{r}}(\mathbb{G})$, the image of the GNS representation associated to the Haar state; and a maximal $\mathrm{C}^*$-norm completion, the universal algebra of continuous functions, $C_{\text{u}}(\mathbb{G})$. The compact quantum group $\mathbb{G}$ is \emph{coamenable} if $\mathcal{O}(\mathbb{G})$ has a unique $\mathrm{C}^*$-norm completion to an algebra of continuous functions on a compact quantum group, and so in particular $C_{\text{r}}(\mathbb{G})\cong C_{\text{u}}(\mathbb{G})$. The Haar state is faithful on $\mathcal{O}(\mathbb{G})$ and $C_{\text{r}}(\mathbb{G})$, but $C_{\text{r}}(\mathbb{G})$ does not in general admit a character. On the other hand, $C_{\text{u}}(\mathbb{G})$ does admit a character, but the Haar state is no longer faithful in general.

\bigskip

After an abelianisation $\pi_{\text{ab}}:C(\mathbb{G})\to C(\mathbb{G})/N_{\text{ab}}$, and via Gelfand's theorem, the algebra of continuous functions on the \emph{classical version} of a compact quantum group is  given by the algebra of continuous function on the set of characters. However, not  every completion $C_\alpha(\mathbb{G})$ of $\mathcal{O}(\mathbb{G})$ admits a classical version: in particular, when $\mathbb{G}$ is not coamenable the abelianisation of $C_{\text{r}}(\mathbb{G})$ is zero, and $C_{\text{r}}(\mathbb{G})$ admits no characters. This work includes a study of the classical versions of  quantum permutation groups $\mathbb{G}\leq S_N^+$, and working at the universal level ensures that talking about the classical version $G\leq \mathbb{G}$ makes sense.

      \bigskip

The quantum subgroup relation $\mathbb{H}\leq \mathbb{G}$ is given at the universal level: a quantum subgroup is given by a surjective $\ast$-homomorphism $\pi: C_{\text{u}}(\mathbb{G})\to C_{\text{u}}(\mathbb{H})$ that respects the comultiplication in the sense that:
$$\Delta_{C_{\text{u}}(\mathbb{H})}\circ\pi=(\pi\otimes \pi)\circ \Delta.$$
Every such morphism of algebras of continuous function  $C_{\text{u}}(\mathbb{G})\to C_{\text{u}}(\mathbb{H})$ restricts to a  morphism on the level of regular functions $\mathcal{O}(\mathbb{G})\to \mathcal{O}(\mathbb{H})$; and every  morphism $\mathcal{O}(\mathbb{G})\to \mathcal{O}(\mathbb{H})$ extends to the level of universal algebras of continuous functions \cite{bmt}.

\bigskip

Key in this work is the notion of a \emph{quasi-subgroup} $\mathbb{S}_{\phi}\subseteq \mathcal{S}(C_{\text{u}}(\mathbb{G}))$, defined as the set of states $\varphi$  that are absorbed by a given idempotent state $\phi$ on $C_{\text{u}}(\mathbb{G})$:
        $$\varphi\star \phi=\phi=\phi\star \varphi.$$
       If $h_{\mathbb{H}}:=h_{C_\alpha (\mathbb{H})}\circ \pi$ is a Haar idempotent associated with $\pi:C(\mathbb{G})\to C_\alpha(\mathbb{H})$, it is the case that
$$\{\varphi\circ \pi\,\colon\,\varphi\in\mathcal{S}(C_{\alpha}(\mathbb{H}))\}\subseteq \mathbb{S}_{h_{\mathbb{H}}}.$$
\begin{remark}\label{Vaes}
As explained by Stefaan Vaes  \cite{Vaes}, in general this is not an equality. In particular the Haar state of $C_{\text{r}}(\mathbb{G})$ in $C_{\text{u}}(\mathbb{G})$,
$$h_{\text{r}}:=h_{C_{\text{r}}(\mathbb{G})}\circ \pi_{\text{r}},$$
is in fact equal to the Haar state on $C_{\text{u}}(\mathbb{G})$. Thus the quasi-subgroup generated by $h_{\text{r}}$ is the whole state space of $C_{\text{u}}(\mathbb{G})$, but in the non-coamenable case there are states on $C_{\text{u}}(\mathbb{G})$, such as the counit, that do not factor through $\pi_r$, and thus in this case:
$$\{\varphi\circ \pi_{\text{r}}\,\colon\,\varphi\in \mathcal{S}(C_{\text{r}}(\mathbb{G}))\}\subsetneq \mathbb{S}_{h_{\text{r}}}.$$
 Vaes goes on to prove that in the universal case of $\pi: C(\mathbb{G})\to C_{\text{u}}(\mathbb{H})$ that indeed:
\begin{equation}\{\varphi\circ \pi\,\colon\,\varphi\in\mathbb{H}\}=\mathbb{S}_{h_{\mathbb{H}}},\label{universalHaar}\end{equation}
and this is more satisfactory for a theory of quasi-subgroups.  Note that Vaes's observation yields Theorem \ref{absorb} as a special case. It is believed that (\ref{universalHaar}) is not in the literature, however as its proof requires representation theory, not used in the current work, Vaes's proof is left to an appendix.
\end{remark}

\textbf{From this point on, all algebras of continuous functions will be assumed universal,} $C(\mathbb{G})\cong C_{\text{u}}(\mathbb{G})$. Careful readers can extract results which hold more generally.

\subsection{Quantum Permutation Groups}
Let $C(\mathbb{X})$ be a $\mathrm{C}^*$-algebra with unit $\mathds{1}_{\mathbb{X}}$. A (finite) \emph{partition of unity} is a (finite) set of projections $\{p_i\}_{i=1}^N\subset C(\mathbb{X})$ that sum to the identity:
$$\sum_{i=1}^Np_i=\mathds{1}_{\mathbb{X}}.$$
\begin{definition}
A \emph{magic unitary} is a matrix $u\in M_N(C(\mathbb{X}))$ such that the rows and columns are partitions of unity:
$$\sum_{k=1}^Nu_{ik}=\mathds{1}_{\mathbb{X}}=\sum_{k=1}^Nu_{kj}\qquad (1\leq i,j\leq N).$$
\end{definition}
Consider the universal unital $\mathrm{C}^*$-algebra:
$$C(S_N^+):=\mathrm{C}^*(u_{ij}\,\colon\, u \text{ an $N\times N$ magic unitary}).$$
Define
\begin{equation}\Delta(u_{ij})=\sum_{k=1}^Nu_{ik}\otimes u_{kj}.\label{rep}\end{equation}
Using the universal property, Wang \cite{wa2} shows that $\Delta$ is a $\ast$-homomorphism, and $S_N^+$ is a compact quantum group, called \emph{the} quantum permutation group on $N$ symbols. Note $S_N^+$ is not coamenable for $N\geq 5$ \cite{ba3}.
\begin{definition}
 Let $\mathbb{G}$ be a compact quantum group. A magic unitary $u\in M_N(C(\mathbb{G}))$  whose entries generate $C(\mathbb{G})$ as a $\mathrm{C}^{*}$-algebra, and such that $\Delta(u_{ij})$ is given by (\ref{rep}), is called a \emph{magic fundamental representation}. A compact quantum group that admits such a \emph{magic fundamental representation} is known as a quantum permutation group, and by the universal property $\mathbb{G}\leq S_N^+$.
\end{definition}

The relation $\mathbb{G}\leq S_N^+$ yields a specific fundamental magic representation $u\in M_N(C(\mathbb{G}))$, and whether $u_{ij}$ is a generator of $C(\mathbb{G})$ or of $C(S_N^+)$ should be clear from  context. \textbf{From this point on, all quantum groups $\mathbb{G}$ will be assumed to be quantum permutations groups $\mathbb{G}\leq S_N^+$}. Again, careful readers can extract results which hold more generally.

\bigskip

The antipode is given by:
$$S(u_{ij})=u_{ji}\implies S^2(u_{ij})=u_{ij},$$
that is quantum permutation groups are of  Kac type.
\begin{proposition}\label{reverse}
Let $\varphi_1,\varphi_2$ be states on $C(\mathbb{G})$:
$$(\varphi_1\star\varphi_2)\circ S=(\varphi_2\circ S)\star (\varphi_1\circ S).$$
\end{proposition}
\begin{proof}
Where $\tau$ is the flip, $f\otimes g\mapsto g\otimes f$, in $\mathcal{O}(\mathbb{G})$:
$$\Delta\circ S=(S\otimes S)\circ \tau\circ \Delta.$$
If $f\in\mathcal{O}(\mathbb{G})$, then using the antipodal property
\begin{align*}
  ((\varphi_1\star\varphi_2)\circ S)(f) & =((\varphi_2\circ S)\star (\varphi_1\circ S))(f).
\end{align*}
The same holds for all $f\in C(\mathbb{G})$ because the antipode is bounded, and the comultiplication is a $\ast$-homomorphism, and thus both are norm-continuous.
\end{proof}

\begin{lemma}[Section 3, \cite{frs}]\label{idemS}
If a state $\phi$ on $C(\mathbb{G})$ is idempotent, $\phi\star\phi=\phi$, then $\phi\circ S=\phi$.
\end{lemma}

\subsection{The Bidual}\label{bid}
In the sequel the \emph{bidual} $C(\mathbb{X})^{**}$ of a unital $\mathrm{C}^{*}$-algebra $C(\mathbb{X})$ will be used. Here some of its properties are summarised from Takesaki, Vol. I. \cite{ta1}. The bidual admits $C(\mathbb{X})^{*}$ as a predual, and so is a von Neumann algebra. States $\varphi$ on $C(\mathbb{X})$ have extensions to states $\omega_\varphi$ on $C(\mathbb{X})^{**}$.   The \emph{support projection} $p_\varphi\in C(\mathbb{X})^{**}$ of a state $\varphi$ on $C(\mathbb{X})$ is the smallest projection $p$ such that $\omega_\varphi(p)=1$: if $\omega_\varphi(p)=1$ then $\varphi$ is said to be \emph{supported on $p$}, and   $p_\varphi\leq p$. It has the property that:
$$\varphi(f)=\omega_{\varphi}(fp_\varphi)=\omega_{\varphi}(p_\varphi f)=\omega_{\varphi}(p_\varphi fp_\varphi)\qquad (f\in C(\mathbb{X})).$$
If $N\subseteq C(\mathbb{X})$ is an ideal, then $N^{**}\subseteq C(\mathbb{X})^{**}$ is $\sigma$-weakly closed, and so equal to $C(\mathbb{X})^{**}q$ for a central projection $q\in C(\mathbb{X})^{**}$. Then, as $\mathrm{C}^*$-algebras:
\begin{equation}
C(\mathbb{X})^{**}\cong\left(C(\mathbb{X})/N\right)^{**}\oplus N^{**}.\label{directsum}
\end{equation}

The embedding $C(\mathbb{X})\subseteq \mathbb{C}(\mathbb{X})^{**}$ is an isometry, so that $C(\mathbb{X})$ is norm closed, and the norm closure of a norm dense $\ast$-subalgebra $\mathcal{O}(\mathbb{X})\subseteq C(\mathbb{X})$ in $C(\mathbb{X})^{**}$ is $C(\mathbb{X})$. In addition, the $\sigma$-weak closures of $\mathcal{O}(\mathbb{X})$ and $C(\mathbb{X})$ are both $C(\mathbb{X})^{**}$. A $\ast$-homomorphism $T:C(\mathbb{X})\to C(\mathbb{Y})$ extends to a $\sigma$-weakly continuous $\ast$-homomorphism:
$$T^{**}:C(\mathbb{X})^{**}\to C(\mathbb{Y})^{**}.$$
In particular, the extension of a character on $C(\mathbb{X})$ is a character on $C(\mathbb{X})^{**}$. The product on the bidual is separately $\sigma$-weakly continuous:
$$\left(\lim_\lambda f_\lambda\right)f=\lim_\lambda(f_\lambda f)\qquad (f_\lambda,f\in C(\mathbb{X})^{**}).$$

Via the Sherman--Takeda Theorem \cite{she,tak}, projections $p_1,\dots,p_N\in C(\mathbb{X})$ may be viewed as Hilbert space projections. Then
\begin{equation}\lim_{n\to \infty}[(p_1\cdots p_N)^n]=p_1\wedge\cdots\wedge p_N,\label{altern}\end{equation}
strongly \cite{hal}. The powers of products of projections are in the unit ball. The strong and $\sigma$-strong topologies coincide on the unit ball, and $\sigma$-strong convergence implies $\sigma$-weak convergence of (\ref{altern}). Finally, for any Borel set $E\subseteq \sigma(f)$ of self-adjoint $f\in C(\mathbb{X})$, the spectral projection $\mathds{1}_E(f)\in C(\mathbb{X})^{**}$.

\section{Pal sets and quasi-subgroups\label{palquasi}}
\subsection{Pal sets}
The following notation/terminology is outlined in \cite{mcc} and used hereafter:
\begin{definition}\label{gelbir}
Given a quantum permutation group $\mathbb{G}$,
the \emph{Gelfand--Birkhoff picture} interprets elements of the state-space as quantum permutations, so that $\varphi\in \mathbb{G}$ means $\varphi$ is a state on $C(\mathbb{G})$, and $\mathbb{S}\subseteq \mathbb{G}$ denotes a subset of the state space $\mathcal{S}(C(\mathbb{G}))$.
\end{definition}

\begin{definition}
A subset $\mathbb{S}\subseteq \mathbb{G}$ is \emph{closed under convolution} if
$$\varphi,\rho\in \mathbb{S}\implies \varphi\star\rho\in\mathbb{S};$$
it is \emph{closed under reverses} if
$$\varphi\in\mathbb{S}\implies (\varphi\circ S)\in\mathbb{S};$$
it \emph{contains the identity} if $C(\mathbb{G})$ admits a counit $\varepsilon$, and $\varepsilon\in\mathbb{S}$.
\end{definition}
\begin{proposition}
Suppose that $\pi:C(\mathbb{G})\to C(\mathbb{H})$ gives a (closed) quantum subgroup $\mathbb{H}\leq \mathbb{G}$. Then the set:
$$\mathbb{H}^{\leq \mathbb{G}}:=\{\varphi\circ \pi\,\colon\,\varphi\in \mathbb{H}\},$$
is closed under convolution, closed under reverses, and contains the identity.
\end{proposition}

There are subsets $\mathbb{S}\subset \mathbb{G}$ that are closed under convolution, closed under reverses, and contain the identity that are \emph{not} associated with quantum subgroups in this way.

\begin{example}
Let $\Gamma$ be a finite group with a non-normal subgroup $\Lambda\leq \Gamma$. The state space of $C(\widehat{\Gamma})$, denoted here $\widehat{\Gamma}$, is the set of positive-definite functions on $\Gamma$. Define:
\begin{equation}\mathbb{S}_{\Lambda}=\{\varphi\in\widehat{\Gamma}:\varphi(\lambda)=1\text{ for all }\lambda\in \Lambda\}.\label{SLAM}\end{equation}
The convolution for states on $C(\widehat{\Gamma})$ is pointwise multiplication, therefore $\mathbb{S}_{\Lambda}$ is closed under convolution. The reverse of $\varphi\in \widehat{\Gamma}$ is:
$$(\varphi\circ S)(\gamma)=\varphi(\gamma^{-1}),$$
and $\Lambda$ is a group so $\mathbb{S}_{\Lambda}$ is closed under reverses. The identity, $\mathds{1}_{\Gamma}\in \mathbb{S}_{\Lambda}$.
\end{example}
\begin{example}\label{KacP}(\cite{mcc})
Let $G_0$ be the Kac--Paljutkin quantum group with algebra of functions
$$C(G_0)=\mathbb{C}f_1\oplus \mathbb{C}f_2\oplus\mathbb{C}f_3\oplus\mathbb{C}f_4\oplus M_2(\mathbb{C}).$$
Where $f^i$ is dual to  $f_i$, and $E^{ij}$ is dual to the matrix unit $E_{ij}$ in the $M_2(\mathbb{C})$ factor, the convex hulls $\operatorname{co}(\{f^1,f^4,E^{11}\})$ and $\operatorname{co}(\{f^1,f^4,E^{22}\})$ are closed under convolution, under reverses, and contain the identity, $\varepsilon=f^1$.
\end{example}
\begin{example}
Let $\mathbb{G}$ be a  quantum permutation group with $u_{ii}\in C(\mathbb{G})$ non-central. Define a subset $\mathbb{G}_i\subset \mathbb{G}$ by:
$$\mathbb{G}_i:=\{\varphi\in \mathbb{G}\,:\,\varphi(u_{ii})=1\}.$$
This set is closed under convolution, and closed under reverses because $S(u_{ii})=u_{ii}$. Finally $\varepsilon\in\mathbb{G}_i$ as $\varepsilon(u_{ij})=\delta_{i,j}$. More in Section  \ref{stab}.
\end{example}
\begin{definition}\label{Pal}
 A \emph{Pal set} is a non-empty convex weak-* closed subset $\mathbb{S}\subseteq \mathbb{G}$ that is closed under convolution.
\end{definition}

\begin{theorem}\label{THQI}
A Pal set $\mathbb{S}\subseteq\mathbb{G}$ contains a unique $\mathbb{S}$-invariant state, $\phi_{\mathbb{S}}\in\mathbb{S}$, such that for all $\varphi\in\mathbb{S}$:
$$\phi_{\mathbb{S}}\star \varphi=\phi_{\mathbb{S}}=\varphi\star\phi_{\mathbb{S}}.$$.
\end{theorem}
\begin{proof}
This has exactly the same proof as Theorem \ref{THQIH}, except rather than considering a $K_\omega$ for each positive linear functional $\omega$ on $C(\mathbb{G})$, the proof uses for each $\omega\in\operatorname{cone}(\mathbb{S})$ the closed set $K_\omega\cap \mathbb{S}$.
\end{proof}

The strength of the notion of a Pal set is that, as will be seen in Section \ref{stab}, they can be easy to describe, and yield idempotent states with certain properties. The problem with Definition \ref{Pal}  is that Pal sets are not in general sub-objects, not state-spaces of algebras of continuous functions on a compact quantum group. It is possible to talk about compact quantum \emph{hypergroups} in this setting \cite{frs,fst,lvd}, but this avenue will not be pursued here. Furthermore, the correspondence $\mathbb{S}\to \phi_{\mathbb{S}}$ is not one-to-one. For example,  the Pal set $\mathbb{H}^{\leq \mathbb{G}}$ yields the Haar idempotent $h_{\mathbb{H}}\circ \pi$. The singleton $\{h_{\mathbb{H}}\circ \pi\}$ is  a Pal set with the same idempotent $h_{\mathbb{H}}\circ \pi$.

\bigskip

Another such non-correspondence occurs for the Pal set of central states:
\begin{definition}
Where
$$\{u^{\alpha}_{ij}\,\colon\,i,j=1,\dots, d_\alpha,\,\alpha\in\operatorname{Irr}(\mathbb{G})\}$$
are matrix coefficients of mutually inequivalent irreducible unitary representations, a  state $\varphi\in\mathbb{G}$ is a \emph{central state} if for all $\alpha\in\operatorname{Irr}(\mathbb{G})$ there exists $\varphi(\alpha)\in\mathbb{C}$ such that:
$$\varphi(u^{\alpha}_{ij})=\varphi(\alpha)\delta_{i,j}.$$
Equivalently, a state $\varphi\in\mathbb{G}$ is a central state if it commutes with all states with respect to convolution:
$$\varphi\star\rho=\rho\star \varphi\qquad(\rho\in\mathbb{G}).$$
\end{definition}
\begin{proposition}
The set of central states is a Pal set with idempotent state $h\in \mathbb{G}$.
\end{proposition}

In \cite{fre}, an $S_N^+$ analogue of the measure on $S_N$ constant on transpositions, a central state $\varphi_{\text{tr}}$ on $C(S_N^+)$, is studied, and it is shown that the convolution powers $(\varphi_{\text{tr}}^{\star k})_{k\geq 0}$ are a sequence of central states converging to the Haar state.

\subsection{Quasi-subgroups}

To fix the non-injectivity of the association of a Pal set $\mathbb{S}$ with an idempotent $\phi_{\mathbb{S}}$ is to define a  \emph{quasi-subgroup}. This nomenclature of \emph{quasi}-subgroup is inspired by Kasprzak and So{\l}tan \cite{kap}.

\begin{proposition}
Given an idempotent state $\phi\in \mathbb{G}$, the set:
  \begin{equation}
  \mathbb{S}_{\phi}:=\{\varphi\in\mathbb{G}\colon \varphi\star\phi=\phi=\phi\star \varphi\}\end{equation}
  is a Pal set   with idempotent state $\phi$.
\end{proposition}
\begin{proof}
 By associativity, $\mathbb{S}_{\phi}$ is closed under convolution. Convexity is straightforward. Weak-* closure follows from the fact that for $\rho\in \mathbb{G}$, the convolution operators $\varphi\mapsto \rho\star \varphi$ and $\varphi\mapsto \varphi\star \rho$ are weak-* continuous.
\end{proof}
\begin{definition}
A \emph{quasi-subgroup} is a subset of the state space of the form $\mathbb{S}_{\phi}$ for an idempotent state $\phi$ on $C(\mathbb{G})$; the \emph{quasi-subgroup generated by $\phi$.}
\end{definition}
The quasi-subgroup $\mathbb{S}_{\phi}$ is the largest Pal set $\mathbb{S}$ such that $\phi$ is the $\mathbb{S}$-invariant state in $\mathbb{S}$, and there is a one-to-one correspondence between quasi-subgroups and idempotent states.

\subsection{Group-like projections}
Group-like projections, and their link with idempotent states, were first noted by Landstad and Van Daele \cite{lvd}. This definition can be extended to the bidual:
\begin{definition}
A group-like projection $p\in C(\mathbb{G})^{**}$ is a non-zero projection such that:
$$\Delta^{**}(p)(\mathds{1}_{\mathbb{G}}\otimes p)=p\otimes p.$$
\end{definition}

 In the finite case, there is a bijective correspondence between idempotent states and group-like projections: every idempotent state has group-like density with respect to the Haar state \cite{frs} (and this group-like density coincides with the support projection \cite{mcc}). In the compact case, continuous group-like projections $p\in C(\mathbb{G})$ with $h(p)>0$ give densities to idempotent states via the Fourier transform, $p\mapsto h(\cdot p)/h(p)$, but the converse does not hold (see Section \ref{exotic} and Corollary \ref{Haartq}). However  it is shown here  that every group-like projection in the \emph{bidual} yields a Pal set, and thus an idempotent state, but as seen in Proposition \ref{conversefail} a converse statement does not hold.   In general, it can only be said that idempotent states are associated with group-like projections in the multiplier algebra of the dual discrete quantum group \cite{frs}. Furthermore this result has been generalised to the locally compact setting by Kasprzak and So{\l}tan \cite{kap}.

\bigskip

The language of wave-function collapse will be used to talk about idempotent states with group-like density, and later illustrate the difference between Haar and non-Haar idempotents:

\begin{definition}\label{condition}
Let $q\in C(\mathbb{G})^{**}$ be a projection and $\varphi\in \mathbb{G}$. If $\omega_{\varphi}(q)>0$, then  \emph{$\varphi$ conditioned by $q=1$} is given by:
$$\widetilde{q}\varphi(g):=\frac{\omega_\varphi(qg q)}{\omega_\varphi(q)}\qquad (g\in C(\mathbb{G})),$$
and  $\varphi\to \widetilde{q}\varphi$ is referred to as \emph{wave-function collapse}. Furthermore, say that a subset $\mathbb{S}\subseteq \mathbb{G}$ is \emph{stable under wave-function collapse} if for all projections $q\in C(\mathbb{G})^{**}$,
\begin{equation}(\varphi\in\mathbb{S}\text{ and }\omega_{\varphi}(q)>0)\implies \widetilde{q}\varphi\in \mathbb{S}.\label{wfc}\end{equation}
\end{definition}

The following is well known in the algebraic setting  (Prop. 1.8, \cite{lvd}), and a similar proof is known to work in the finite quantum group setting (Cor. 4.2, \cite{frs}). For the benefit of the reader, the proof is reproduced in the current setting:
\begin{proposition}\label{prop8.3}
If $p\in C(\mathbb{G})$ is a continuous group-like projection such that $h(p)>0$,  then $\widetilde{p}h\in\mathbb{G}$ is an idempotent state.
\end{proposition}
\begin{proof}Let $\phi=\widetilde{p}h$. The difference between $\omega_h$ and $h$ can be suppressed here as ${\omega_h}_{\left.\right|_{C(\mathbb{G})}}=h$. Let $f\in \mathcal{O}(\mathbb{G})$:
\begin{align*}
  (\phi\star\phi)(f) & =\frac{1}{h(p)^2}\sum h(pf_{(1)}p)h(pf_{(2)}p)
 =\frac{1}{h(p)^2}\sum h(f_{(1)}p)h(f_{(2)}p) \\
   & =\frac{1}{h(p)^2}(h\otimes h)\left(\Delta(f)(p\otimes p)\right)  =\frac{1}{h(p)^2}(h\otimes h)\left(\Delta(f)\Delta(p)(\mathds{1}_{\mathbb{G}}\otimes p)\right)
   \\&=\frac{1}{h(p)^2}(h\otimes h)\left(\Delta(fp)(\mathds{1}_{\mathbb{G}}\otimes p)\right)  =\frac{1}{h(p)^2}h(fp)h(p)=\frac{h(pfp)}{h(p)}=\phi(f),
\end{align*}
where the traciality of the Haar state, $p^2=p$, and $(h\otimes \varphi)(\Delta(f)(\mathds{1}_{\mathbb{G}}\otimes g))=h(f)\varphi(g)$ (Remark 2.2.2 i., \cite{tim}) were used. By norm-continuity this implies that $\widetilde{p}h$ is idempotent.
\end{proof}

\begin{proposition}\label{prop27}
Let $\phi=\widetilde{p}h$ be an idempotent state with $p\in \mathcal{O}(\mathbb{G})$ a group-like projection. Then
  $$\mathbb{S}_\phi=\{\varphi\in\mathbb{G}\,\colon\,\varphi(p)=1\}.$$
\end{proposition}
\begin{proof}
Suppose that $\varphi(p)=1$.
Similarly to the proof of Proposition \ref{prop8.3}, for $f\in \mathcal{O}(\mathbb{G})$:
    \begin{equation}(\phi\star \varphi)(f)=\frac{1}{h(p)}(h\otimes \varphi)(\Delta(fp_{\phi})(\mathds{1}_{\mathbb{G}}\otimes p_{\phi}))=\frac{h(fp_{\phi})}{h(p)}\varphi(p)=\phi(f),\label{eq1}\end{equation}
    and by weak-* continuity, $\phi\star \varphi=\phi$. To show $\varphi\star \phi=\phi$, conduct a similar calculation, using $\Delta(p)(p\otimes \mathds{1}_{\mathbb{G}})=p\otimes p$ (Prop. 1.6, \cite{lvd}), and $(\varphi\otimes h)(\Delta(f)(g\otimes \mathds{1}_{\mathbb{G}}))=\varphi(g)h(f)$ (Remark 2.2.2 i., \cite{tim}).

    \bigskip

     On the other hand, suppose that $\varphi\in \mathbb{S}_\phi$, so that $\phi\star\varphi=\phi$.
Considering the equality of the first and third quantities of (\ref{eq1})  at $f=p$, with the existence of $\widetilde{p}h$ implying $h(p)>0$:
   \begin{align*}
     (\phi\star \varphi)(p)=\frac{h(pp_\phi)}{h(p)}\varphi(p)=\phi(p_\phi)\varphi(p)=\varphi(p).
   \end{align*}
If $\phi\star \varphi=\phi$, then $\varphi(p)=\phi(p)=1$.
   \end{proof}

\begin{proposition}\label{support}
  If states $\varphi_1,\varphi_2$ on $C(\mathbb{G})$   are supported on a group-like projection $p\in C(\mathbb{G})^{**}$, then so is $\varphi_1\star \varphi_2$.
\end{proposition}
\begin{proof}
The proof for the finite case (Prop. 3.12, \cite{mc1}) applies with some adjustments.
 Let $(p^\lambda)\subset \mathcal{O}(\mathbb{G})$ converge $\sigma$-weakly to $p\in C(\mathbb{G})^{**}$. As the extension of $\Delta$ to $\Delta^{**}$ is $\sigma$-weakly continuous
\begin{align*}
  \lim_\lambda\left[\Delta(p^\lambda)\right](1\otimes p) & =p\otimes p
  \end{align*}
The product is separately continuous, and $\omega_{\varphi_1}\otimes\omega_{\varphi_2}$ is $\sigma$-weakly continuous.
  \begin{align*}
  \implies \lim_\lambda(\omega_{\varphi_1}\otimes\omega_{\varphi_2})\sum p^{\lambda}_{(0)}\otimes p^\lambda_{(1)}p & =(\omega_{\varphi_1}\otimes\omega_{\varphi_2})(p\otimes p) \\
  \implies \lim_\lambda \sum \omega_{\varphi_1}( p^{\lambda}_{(0)})\omega_{\varphi_2}(p^\lambda_{(1)}p) & =1
  \end{align*}
   Note that as $\varphi_2$ is supported on $p$:
  \begin{align*}
 \implies \lim_\lambda \sum \varphi_1( p^{\lambda}_{(0)})\varphi_2(p^\lambda_{(1)}) & =1
   \\ \implies \lim_\lambda (\varphi_1\star\varphi_2)(p^\lambda)&=1
   \\ \implies \lim_\lambda \omega_{\varphi_1\star\varphi_2}(p^\lambda)= \omega_{\varphi_1\star\varphi_2}(p)&=1.
\end{align*}

\end{proof}

\begin{proposition}\label{Palidemp}
Suppose $p\in C(\mathbb{G})^{**}$ is a group-like projection. Then:
$$\mathbb{S}_p:=\{\varphi\in\mathbb{G}\,\colon\,\omega_\varphi(p)=1\},$$
is a Pal set, and so there is an idempotent $\phi$ supported on $p$ such that $p_\phi\leq p$.
\end{proposition}
\begin{proof}
First $\mathbb{S}_p$ is non-empty because $p$ is normal and as $\|p\|_{C(\mathbb{G})^{**}}=1$,  there exists a state $\omega$ on $C(\mathbb{G})^{**}$ such that $\omega(p)=1$ \cite{mur}, whose restriction to $C(\mathbb{G})$ is a state in $\mathbb{S}_p$. Weak-* closure and convexity are straightforward, and closure under convolution follows from Proposition \ref{support}.
\end{proof}

It is not claimed that $p$ is  necessarily equal to the support projection of the idempotent state in $\mathbb{S}_p$; and in the below it is not claimed that the idempotent state in $\mathbb{S}_p$ is necessarily equal to $\phi$.
\begin{theorem}\label{pqprops}
Suppose that an idempotent state $\phi\in\mathbb{G}$  has group-like support projection $p\in C(\mathbb{G})^{**}$. Then the quasi-subgroup generated by $\phi$:
$$\mathbb{S}_\phi\subseteq \mathbb{S}_p.$$
\end{theorem}
\begin{proof}
Consider $\varphi\in \mathbb{S}_\phi$ not supported on $p$ so that $\omega_{\varphi}(p)<1$.
Using similar notation and techniques to Proposition \ref{support}, apply the $\sigma$-weakly continuous $\omega_{\varphi}\otimes \omega_{\phi}$ to both sides of $\Delta^{**}(p)(\mathds{1}_{\mathbb{G}}\otimes p)=p\otimes p$, using the fact that $p$ is the support projection of $\phi$:
  \begin{align*}
    \implies \lim_\lambda \left(\sum \omega_{\varphi}(p^\lambda_{(0)})\otimes \omega_{\phi}(p^\lambda_{(1)}p)\right) & =\omega_{\varphi}(p)\otimes \omega_{\phi}(p)\\
     \implies \lim_\lambda \left(\sum \varphi(p^\lambda_{(0)})\otimes \omega_{\phi}(p^\lambda_{(1)})\right) & <1\\
      \implies \lim_\lambda \left(\sum \varphi(p^\lambda_{(0)})\otimes \phi(p^\lambda_{(1)})\right) & <1\\
        \implies \lim_\lambda \left((\varphi\star \phi)(p^\lambda)\right) & <1\\
        \implies \lim_\lambda \left(\phi(p^\lambda)\right) & <1
        \\ \implies \omega_{\phi}(p)&<1,
  \end{align*}
a nonsense,  and so $\omega_\varphi(p)=1$.
\end{proof}

It is not the case that every idempotent state $\phi$  has group-like support projection $p_\phi\in C(\mathbb{G})^{**}$. Nor does Theorem \ref{pqprops} hold more generally:

\begin{corollary}\label{conversefail}
Suppose $\mathbb{G}$ is non-coamenable. Then the support projection $p_h\in C(\mathbb{G})^{**}$ of the Haar state is not a group-like projection. Furthermore:
$$\mathbb{S}_{p_h}\subsetneq \mathbb{S}_h.$$

\end{corollary}

\begin{proof}
Assume that the support $p_h\in C(\mathbb{G})^{**}$ is a group-like projection. As $\omega_h(\mathds{1}_{\mathbb{G}})=1$, $\mathds{1}_{\mathbb{G}}-p_h>0$ strictly as $\mathbb{G}$ is at the universal level and $\mathbb{G}$ is assumed non-coamenable. Therefore there exists a state $\omega_\varphi$ on $C(\mathbb{G})^{**}$ such that
$$\omega_\varphi(\mathds{1}_{\mathbb{G}}-p_h)=1\implies \omega_\varphi(p_h)=0.$$
Restrict $\omega_\varphi$ to a state $\varphi$ on $C(\mathbb{G})$. By Theorem \ref{pqprops} it follows that $\varphi$ is not invariant under the Haar state, which is absurd as $\mathbb{S}_h=\mathbb{G}$.
\end{proof}
There is a group-like projection $p$ such that
$$\mathbb{S}_p=\mathbb{S}_h;$$
the unit $p=\mathds{1}_{\mathbb{G}}$.

\bigskip

Note, via the following, where $|f|:=\sqrt{f^*f}$, there is a relationship between  quantum subgroups and wave-function collapse:
\begin{proposition}[Th. 3.3, \cite{fst}]\label{fst3.3}
Let $\mathbb{G}$ be a compact quantum group and $\phi\in C(\mathbb{G})^*$ an idempotent state. Then $\phi$ is a Haar idempotent if and only if the null-space $$N_\phi = \{f\in C(\mathbb{G}) \,:\, \phi(|f|^2)=0\}$$ is a two-sided ideal.
\end{proposition}
As a consequence, an idempotent state is a Haar idempotent if and only if it is tracial. Note in the below  $\omega_{\varphi_0}$ is the extension of the state $\varphi_0$ on $C(\mathbb{H})$ to a state on $C(\mathbb{H})^{**}$.
\begin{lemma}\label{fact}
Suppose that $\mathbb{H}\leq\mathbb{G}$ via $\pi:C(\mathbb{G})\to C(\mathbb{H})$. Then the extension of $\varphi_0\circ \pi$ to a state on $C(\mathbb{G})^{**}$ is given by  $\omega_{\varphi_0}\circ \pi^{**}$.

\end{lemma}
\begin{proof}

The result follows from the $\sigma$-continuity of the maps involved, and ${\pi^{**}}_{\left.\right|_{C(\mathbb{G})}}=\pi$.

\end{proof}
Note that part (i) of the below is restricted to Haar idempotents coming from Haar states on universal versions.
\begin{theorem}\label{wfct}
Suppose that $\phi$ is an idempotent state on $C(\mathbb{G})$.
  \begin{enumerate}
    \item[(i)] If $\phi$ is a (universal) Haar idempotent, then $\mathbb{S}_\phi$ is closed under wave-function collapse.
    \item[(ii)] If $\phi$ is a non-Haar idempotent with group-like projection support, then $\mathbb{S}_\phi$ is not closed under wave-function collapse.
  \end{enumerate}
\end{theorem}
\begin{proof}
\begin{enumerate}
\item[(i)] Suppose $\phi$ is a (universal) Haar idempotent via $\pi:C(\mathbb{G})\to C(\mathbb{H})$. By Vaes's Remark \ref{Vaes}, every $\varphi\in \mathbb{S}_\phi$ is of the form $\varphi=\varphi_0\circ \pi$ for a state $\varphi_0$ on $C(\mathbb{H})$. Suppose $\varphi$ undergoes wave-function collapse to $\widetilde{q}\varphi$. Then, using Lemma \ref{fact}
$$\omega_\varphi(q)>0\implies \omega_{\varphi_0}(\pi^{**}(q))>0\qquad (\omega_{\varphi_0}\in\mathcal{S}(C(\mathbb{H})^{**})).$$ Using Lemma \ref{fact} again, it can be shown that $\widetilde{q}\varphi=\psi\circ \pi$, where:
$$\psi(g)=\frac{\omega_{\varphi_0}(\pi^{**}(q)g\pi^{**}(q))}{\omega_{\varphi_0}(\pi^{**}(q))}\qquad (g\in C(\mathbb{H}),\,\omega_{\varphi_0}\in\mathcal{S}(C(\mathbb{H})^{**})).$$
Thus, again by Vaes's remark, $\psi\circ \pi$ and thus $\widetilde{q}\varphi\in \mathbb{S}_{\phi}$, that is $\mathbb{S}_{\phi}$ is closed under wave-function collapse.
\item[(ii)] Suppose $\phi$ is a non-Haar idempotent with group-like support projection $p\in C(\mathbb{G})^{**}$. By Theorem \ref{pqprops}
$$\mathbb{S}_\phi\subseteq \mathbb{S}_p.$$
 As $\phi$ is a non-Haar idempotent, the $\sigma$-weak closure $\overline{N_\phi}^{\sigma\text{-}w}=C(\mathbb{G})^{**}q$ is only a left ideal, with $q=\mathds{1}_{\mathbb{G}}-p$ non-central. Suppose that for all $u_{ij}\in C(\mathbb{G})$, $u_{ij}q u_{ij}\in \overline{N_\phi}^{\sigma\text{-}w}$. Then $u_{ij}q u_{ij}=u_{ij}q u_{ij}q\implies u_{ij}q u_{ij}=u_{ij}q u_{ij}q u_{ij}$, so that $u_{ij}q u_{ij}$ is a projection. This implies, because $[u_{ij},q]^3=0$ and $[u_{ij},q]$ is skew adjoint, that $u_{ij}q=q u_{ij}$. Therefore $q$ is central and $N_\phi$ is an ideal. Therefore there exists $u_{ij}$ such that $u_{ij}q u_{ij}\not \in \overline{N_\phi}^{\sigma\text{-}w}$:
$$\omega_\phi(|u_{ij}q u_{ij}|^2)>0.$$
By Cauchy--Schwarz:
$$0<\omega_\phi(|u_{ij}q u_{ij}|^2)\leq \omega_\phi(u_{ij}q u_{ij})\leq \omega_\phi(u_{ij}),$$
$$\implies\widetilde{u_{ij}}\phi(q)=\frac{\omega_\phi(u_{ij}q u_{ij})}{\omega_\phi(u_{ij})}>0\implies \widetilde{u_{ij}}\phi(p)<1\implies  \widetilde{u_{ij}}\phi\not\in\mathbb{S}_\phi.$$

\end{enumerate}
\end{proof}

\section{Stabiliser quasi-subgroups}\label{stab}
 The analysis here is helped somewhat by defining the \emph{Birkhoff slice}, a map $\Phi$ from the state space of the algebra of continuous functions $C(\mathbb{G})$ on a quantum permutation group $\mathbb{G}$ to the doubly stochastic matrices:
$$\Phi(\varphi):=(\varphi(u_{ij}))_{i,j=1}^N.$$

Given a finite group $G\leq S_N$ and a partition $\mathcal{P}=B_1\sqcup\cdots \sqcup B_k$ of $\{1,\dots,N\}$, the $\mathcal{P}$-stabiliser subgroup of $G$ can be formed:
$$G_{\mathcal{P}}=\{\sigma\in G\,\colon\, \sigma(B_p)=B_p,\,1\leq p\leq k\}.$$
The spectre of stabiliser quantum subgroups can be seen in the paper of Wang (concluding remark (2), \cite{wa2}), and Huichi Huang has developed the concept in an unpublished work. A $\mathcal{P}$-stabiliser quasi-subgroup of $\mathbb{G}$ can also be defined. There are two, equivalent, definitions. The first definition uses the equivalence relation $\sim_{\mathcal{P}}$ associated to the partition:
$$\mathbb{G}_{\mathcal{P}}:=\{\varphi\in\mathbb{G}\colon \varphi(u_{ij})=0\text{ for all }i\not\sim_{\mathcal{P}}j\}.$$
Alternatively, consider the Birkhoff slice $\mathcal{S}(C(\mathbb{G}))\to M_N(\mathbb{C})$. By relabelling if necessary, the blocks of a partition can be assumed to consist of consecutive labels. Define:
$$\mathbb{G}_{\mathcal{P}}:=\{\varphi\in\mathbb{G}\,\colon\,\Phi(\varphi)\text{ is block diagonal with pattern }\mathcal{P}\},$$
that is:
$$\varphi\in \mathbb{G}_{\mathcal{P}}\iff \Phi(\varphi)=\begin{bmatrix}
                                                  \Phi_{B_1}(\varphi) & 0 & \cdots & 0 \\
                                                  0 & \Phi_{B_2}(\varphi) & \cdots & 0 \\
                                                  \vdots & \vdots & \ddots & \cdots \\
                                                  0 & 0 & \cdots & \Phi_{B_k}(\varphi)
                                                \end{bmatrix},$$
                                                where $\Phi_{B_p}(\varphi)=[\varphi(u_{ij})]_{i,j\in B_p}$.

\begin{theorem}
For any partition $\mathcal{P}$ of $\{1,\dots,N\}$, $\mathbb{G}_{\mathcal{P}}$ is a quasi-subgroup.
\end{theorem}
\begin{proof}
That $\mathbb{G}_{\mathcal{P}}$ is convex, weak-* closed, and closed under convolution is straightforward (using, for example that the Birkhoff slice is multiplicative $\Phi(\varphi_1\star \varphi_2)=\Phi(\varphi_1)\Phi(\varphi_2)$). The universal version gives $\varepsilon\in \mathbb{G}_{\mathcal{P}}$ so that $\mathbb{G}_{\mathcal{P}}$ is non-empty, and so a Pal set.

\bigskip

Suppose that $\phi_{\mathcal{P}}$ is the associated idempotent. Therefore by Lemma \ref{idemS}:
$$\phi_{\mathcal{P}}(u_{ij})=(\phi_{\mathcal{P}}\circ S)(u_{ij})=\phi_{\mathcal{P}}(u_{ji}).$$ For any fixed $j\in \{1,2,\dots,N\}$, there exists $i\in \{1,2,\dots,N\}$ such that $\phi_{\mathcal{P}}(u_{ji})>0$. From here:
$$\phi_{\mathcal{P}}(u_{jj})=(\phi_{\mathcal{P}}\star \phi_{\mathcal{P}})(u_{jj})=\phi_{\mathcal{P}}(u_{ji})\phi_{\mathcal{P}}(u_{ij})+\sum_{k\neq i}\phi_{\mathcal{P}}(u_{jk})\phi_{\mathcal{P}}(u_{kj})>0.$$
 To show that $\mathbb{G}_{\mathcal{P}}$ is equal to
 $$\mathbb{S}_{\phi_{\mathcal{P}}}=\{\varphi\in\mathbb{G}\,\colon\,\varphi\star \phi_{\mathcal{P}}=\phi_{\mathcal{P}}=\phi_{\mathcal{P}}\star \varphi \},$$ suppose $\varphi\in \mathbb{S}_{\phi_{\mathcal{P}}}$, but $\varphi\not\in \mathbb{G}_{\mathcal{P}}$. That implies there exists $u_{ij}$ such that $\varphi(u_{ij})\neq 0$ with $i\not\sim_{\mathcal{P}}j$. But this gives
$$\phi_{\mathcal{P}}(u_{ij})=(\varphi\star\phi_{\mathcal{P}})(u_{ij})=\varphi(u_{ij})\phi_{\mathcal{P}}(u_{jj})+\sum_{k\neq j}\varphi(u_{ik})\phi_{\mathcal{P}}(u_{kj})>0,$$
a contradiction.
\end{proof}

For the partition $j:=\{j\}\sqcup(\{1,2,\dots,N\}\backslash \{j\})$:
$$\mathbb{G}_{j}=\{\varphi\in \mathbb{G}\,\colon\, \varphi(u_{jj})=1\}.$$
Note for any quantum permutation group $\mathbb{\mathbb{G}}$, and $1\leq j\leq N$, the diagonal element $u_{jj}$ is a polynomial group-like projection:
$$\Delta(u_{jj})(\mathds{1}_{\mathbb{G}}\otimes u_{jj})=\left(\sum_{k=1}^Nu_{jk}\otimes u_{kj}\right)(\mathds{1}_{\mathbb{G}}\otimes u_{jj})=u_{jj}\otimes u_{jj}.$$
By Proposition \ref{prop27},  the associated idempotent state is $h_j:=\widetilde{u_{jj}}h$, that is:
$$h_j(f)=\frac{h(u_{jj}fu_{jj})}{h(u_{jj})}\qquad(f\in C(\mathbb{G})).$$

The below is (almost) a special case of Theorem \ref{wfct}, but included as it uses different proof techniques.
\begin{theorem}
The following are equivalent:
\begin{enumerate}
  \item[(i)] $h_j$ is a Haar idempotent,
  \item[(ii)] $u_{jj}$ is central,
  \item[(iii)] $\mathbb{G}_j$ is stable under wave-function collapse.
\end{enumerate}
\end{theorem}

\begin{proof}
(i) $\implies$ (ii): assume $h_j$ is a Haar idempotent, say equal to $h_{\mathbb{H}}\circ \pi$ where $\pi:C(\mathbb{G})\to C(\mathbb{H})$ is the quotient map. Note that because $h_j(u_{jj})=h_{\mathbb{H}}(\pi(u_{jj}))=1$, and $h_{\mathbb{H}}$ is faithful on $\mathcal{O}(\mathbb{H})$,
$$\mathds{1}_{\mathbb{H}}=\pi(\mathds{1}_{\mathbb{G}})=\sum_{m=1}^N \pi(u_{mj})=\pi(u_{jj}),$$
so that $\pi(u_{jj})=\mathds{1}_{\mathds{H}}$ is central in $C(\mathbb{H})$. Assume that $u_{jj}$ is non-central. Then there exists $u_{kl}\in C(\mathbb{G})$ such that $|u_{kl}u_{jj}-u_{jj}u_{kl}|^2>0$. Expanding:
$$u_{jj}u_{kl}u_{jj}-u_{jj}u_{kl}u_{jj}u_{kl}-u_{kl}u_{jj}u_{kl}u_{jj}+u_{kl}u_{jj}u_{kl}>0.$$
Applying the Haar state, which is faithful on $\mathcal{O}(\mathbb{G})$, and using its traciality yields:
\begin{align*}
  h(u_{jj}u_{kl}u_{jj})&>h(u_{jj}u_{kl}u_{jj}u_{kl}u_{jj})  \\
  \implies h_j(u_{kl}) & >h_j(u_{kl}u_{jj}u_{kl}) \\
  \implies h_{\mathbb{H}}(\pi(u_{kl})) & >h_{\mathbb{H}}(\pi(u_{kl}u_{jj}u_{kl}))=h_{\mathbb{H}}(\pi(u_{kl})\pi(u_{jj})\pi(u_{kl})))
  \\ \implies h_{\mathbb{H}}(\pi(u_{kl}))&>h_{\mathbb{H}}(\pi(u_{kl})\mathds{1}_{\mathds{H}}\pi(u_{kl})))=h_{\mathbb{H}}(\pi(u_{kl})),
\end{align*}
an absurdity, and so $u_{jj}$ is central.

\bigskip

(ii) $\implies$ (i): assume that $u_{jj}$ is central, and
$$N_j:=\{f\in C(\mathbb{G})\,\colon\,h_j(|f|^2)=0\}.$$
If $f\in N_j$ then $h(u_{jj}f^*fu_{jj})=0\implies fu_{jj}\in N_h$, the null-space of the Haar state, so that:
$$N_j=\{f\in C(\mathbb{G})\,\colon\, fu_{jj}\in N_h\}.$$
The rest of the argument is the same as (Th. 4.5, \cite{frs}).

\bigskip

(ii) $\implies$ (iii): assume that $u_{jj}$ is central. If $u_{jj}$ is central in $C(\mathbb{G})$ then it is also central in $C(\mathbb{G})^{**}$. Let $\varphi\in\mathbb{G}_j$ and $q\in C(\mathbb{G})^{**}$ such that $\omega_\varphi(q)>0$. Let $p_\varphi\in C(\mathbb{G})^{**}$ be the support projection of $\varphi$. Note that
$$\omega_\varphi(u_{jj})=\varphi(u_{jj})=1\implies p_\varphi\leq u_{jj}\implies p_\varphi=p_\varphi u_{jj}.$$
Note
\begin{align*}
  \omega_\varphi(qu_{jj}q)  =\omega_\varphi(p_\varphi qu_{jj}q p_\varphi) =\omega_\varphi(p_\varphi u_{jj}qqp_\varphi) =\omega_\varphi(p_\varphi q p_{\varphi})=\omega_\varphi(q).
\end{align*}
It follows that:
$$\widetilde{q}\varphi(u_{jj})=\frac{\omega_\varphi(qu_{jj}q)}{\omega_\varphi(q)}=1\implies \widetilde{q}\varphi\in \mathbb{G}_p.$$
(iii) $\implies$ (ii): assume now that $u_{jj}$ is non-central. Therefore there exists $u_{kl}\in C(\mathbb{G})$ such that:
$$u_{jj}u_{kl}\neq u_{kl}u_{jj}.$$
Represent $C(\mathbb{G})$ with the universal GNS representation $\pi_{\text{GNS}}(C(\mathbb{G}))\subseteq B(\mathsf{H})$. Denote
$$p:=\pi_{\text{GNS}}(u_{jj})\text{ and }q:=\pi_{\text{GNS}}(u_{kl}).$$
As $pq\neq qp$,  there exists a unit vector $x\in \operatorname{ran} p$ that is orthogonal to both  $\operatorname{ran} p\cap \operatorname{ran} q$ and $\operatorname{ran} p\cap \ker q$ (in the notation of ((1), \cite{bos}), $x\in M_0$). Define a state on $C(\mathbb{G})$:
$$\varphi_0(f)=\langle x,\pi_{\text{GNS}}(f)x\rangle\qquad (f\in C(\mathbb{G})).$$
Note that:
$$\varphi_0(u_{jj})=\langle x,p x\rangle=\langle x,x\rangle =1\implies \varphi_0\in \mathbb{G}_j.$$

Furthermore, $\varphi_0(u_{kl})$ cannot be one or zero, because together with $x\in\operatorname{ran}p$
\begin{align*}
 \varphi_0(u_{kl})= \langle x,qx\rangle =1 & \implies x\in\operatorname{ran}q \\
  \varphi_0(u_{kl})=\langle x,qx\rangle =0 & \implies x\in\ker q
\end{align*}
but $x$ is orthogonal to both $\operatorname{ran}p\cap\operatorname{ran}q$ and $\operatorname{ran} p\cap \ker q$ so
$$0<\langle x,qx\rangle<1\implies 0<\varphi_0(u_{kl})<1.$$
Now consider $\varphi=\widetilde{u_{kl}}\varphi_0$:
$$\varphi(f):=\frac{\varphi_0(u_{kl}f u_{kl})}{\varphi_0(u_{kl})}=\frac{\langle qx,\pi_{\text{GNS}}(f)qx \rangle}{\langle qx,qx\rangle}\qquad (f\in C(\mathbb{G})).$$
In particular
\begin{align*}
  \varphi(u_{jj}) & =\frac{\langle qx,pqx\rangle}{\langle qx,qx\rangle}
\end{align*}
Together with $qx\in\operatorname{ran}q$:
\begin{align*}
  \varphi(u_{jj})=1 & \implies qx\in\operatorname{ran}p \\
  \varphi(u_{jj})=0 & \implies qx\in\operatorname{ker}p
\end{align*}
By ((6), \cite{bos}), $qx$ is orthogonal to $\operatorname{ran}p\cap \operatorname{ran}q$ and $\operatorname{ker}p\cap \operatorname{ran} q$, and it follows that:
$$0<\varphi(u_{jj})<1,$$
that is,
$$\varphi_0\in \mathbb{G}_j\text{ but }\widetilde{u_{kl}}\varphi_0\not \in \mathbb{G}_j.$$
\end{proof}

Consider
$$(S_{N}^+)_N:=\{\varphi\in S_N^+\,\colon\,\varphi(u_{NN})=1\}.$$
If $\mathbb{H}$ given by $\pi:C(S_{N}^+)\to C(\mathbb{H})$ is an isotropy subgroup in the sense that $\mathbb{H}\subseteq (S_{N}^+)_N$ and so $\pi(u_{NN})=\mathds{1}_{\mathbb{H}}$, then $\mathbb{H}\leq S_{N-1}^+$ by the universal property. In this way, where $\pi_{N-1}:C(S_N^+)\to C(S_{N-1}^+)$ is the quotient
$$[u_{ij}^{S_N^+}]_{i,j=1}^N\to \begin{bmatrix}
                                  u^{S_{N-1}^+} & 0 \\
                                  0 & \mathds{1}_{S_{N-1}^+}
                                \end{bmatrix},$$
the following is a maximal (set of states on an algebra of continuous functions on a) quantum subgroup in the quasi-subgroup  $(S_{N}^+)_N$
$$(S_{N-1}^+)^{\leq S_N^+}=\{\varphi\circ \pi_{N-1}\,\colon\,\varphi\in S_{N-1}^+\}.$$
In the classical regime, $N\leq 3$, quasi-subgroups are subgroups, and so $(S_{N}^+)_N=(S_{N-1}^+)^{\leq S_N^+}$. However:
\begin{proposition}
The inclusion $(S_{N-1}^+)^{\leq S_N^+}\subset (S_{N}^+)_N$ is proper for $N\geq 4$.
\end{proposition}
\begin{proof}
Note that for any $(\varphi\circ \pi_{N-1})\in (S_{N-1}^+)^{\leq S_N^+}$,
\begin{align*}
  (\varphi\circ \pi_{N-1})(u_{11}u_{2N}u_{11})  =\varphi(\pi_{N-1}(u_{11}u_{2N}u_{11}))=\varphi(\pi_{N-1}(u_{11})\pi_{N-1}(u_{2N})\pi_{N-1}(u_{11}))=0,
\end{align*}
as $\pi_{N-1}(u_{2N})=0$. On the other hand, consider $h_N=\widetilde{u_{NN}}h$, the idempotent in the stabiliser quasi-subgroup $(S_N^+)_N$,
\begin{align*}
  h_N(u_{11}u_{2N}u_{11}) =\frac{h(u_{NN}u_{11}u_{2N}u_{11}u_{NN})}{h(u_{NN})}= \frac{h(|u_{2N}u_{11}u_{NN}|^2)}{h(u_{NN})}.
\end{align*}
Note $|u_{2N}u_{11}u_{NN}|^2\neq 0$ (see \cite{mc2}), and  because $h$ faithful on $\mathcal{O}(S_N^+)$, $h_N(u_{11}u_{2N}u_{11})>0$, and thus $h_N$ is not in $(S_{N-1}^+)^{\leq S_N^+}$.

\end{proof}

Trying to do something for more complicated partitions of $\{1,\dots,N\}$, with an (explicit) idempotent state with a density with respect to $\omega_h$ is in general more troublesome. Consider for example:
$$\mathcal{P}_{i,j}:=(\{1,\dots,N\}\backslash\{i,j\})\sqcup\{i\}\sqcup\{j\}.$$
 The obvious way to fix two points is to work with $p_{i,j}:=u_{ii}\wedge u_{jj}$, an element of $ C(\mathbb{G})^{**}$, and given a quantum permutation $\varphi\in \mathbb{G}$, define a subset of $\mathbb{G}$ by:
$$\mathbb{G}_{i,j}:=\{\varphi\in \mathbb{G}:\omega_\varphi(p_{i,j})=1\}.$$
Note that $\mathbb{G}_{i,j}=\mathbb{G}_i\cap \mathbb{G}_j$. However the following is not in general well defined because  $\omega_h(p_{i,j})$ is not necessarily strictly positive:
$$\phi_{i,j}:=\frac{\omega_h(p_{i,j}\cdot p_{i,j})}{\omega_h(p_{i,j})},$$
For example, consider the dual of the infinite dihedral group with the famous embedding $\widehat{D_\infty}\leq S_4^+$.
Working with alternating projection theory, and noting the Haar state on $C(\widehat{D_\infty})$ is $h(\lambda)=\delta_{\lambda,e}$,
$$\omega_h(p_{1,3})=\lim_{n\rightarrow \infty}h ((u_{11}u_{33})^n)=\lim_{n\to \infty}\frac{1}{4^n}=0.$$
\begin{proposition}
The stabiliser quasi-subgroup $(\widehat{D_\infty})_{1,3}$ is the trivial group.
\end{proposition}
\begin{proof}
Let $\varphi\in(\widehat{D_{\infty}})_{1,3}$ so that $\varphi(u_{11})=\varphi(u_{33})=1$. Then $\Phi(\varphi)=I_4$ and, as will be seen later, by Proposition \ref{character}, $\varphi$ is a character, necessarily equal to the counit.
\end{proof}
By Proposition \ref{charsupport}, $p_{1,3}=p_{\varepsilon}$, the support projection of the counit. As $C(\widehat{D_\infty})$ is coamenable, the Haar state is faithful and so $\omega_h(p_{\varepsilon})=0$ implies that the support projection of the counit, $p_{\varepsilon}$ is not in $C(\widehat{D_\infty})$ (and indeed the universal unital $\mathrm{C}^*$-algebra generated by two projections $p$ and $q$ does not include $p\wedge q$).

\bigskip

 Note that in general $\{\varepsilon\}$ is a quantum subgroup of any quantum permutation group in the sense that $\varepsilon$ is a Haar idempotent via the quotient $\pi:C(\mathbb{G})\to C(e)$ to the trivial group $\{e\}\leq\mathbb{G}$:
$$[u_{ij}]_{i,j=1}^N\to \operatorname{diag}(1_{\mathbb{C}},\dots,1_{\mathbb{C}}).$$

\section{Exotic quasi-subgroups of the quantum permutation group\label{exotic}}
A second reason for studying Pal sets and their generated quasi-subgroups is to postulate, or rather speculate, on, for some $N\geq 4$, the existence of an \emph{exotic} intermediate quasi-subgroup:
$$S_N\subsetneq \mathbb{S}_N\subsetneq S_N^+.$$
It is currently unknown whether or not there is a Haar idempotent giving an exotic intermediate quantum subgroup $S_N\lneq \mathbb{G}_N\lneq S_N^+$ for some  $N\geq 6$. It is the case that $S_N=S_N^+$ for $N\leq 3$, and for $N=4$ \cite{bb3} and $N=5$ \cite{ban} there is no such Haar idempotent. If there is \emph{no} exotic intermediate quasi-subgroup $S_N\subsetneq \mathbb{S}_N \subsetneq S_N^+$ then it is the case that $S_N$ is a maximal quantum subgroup of $S_N^+$ for all $N$. Of course this is stronger than the non-existence of an exotic intermediate quantum subgroup. Indeed it is strictly stronger in the sense that given a quantum permutation group $\mathbb{G}$ and its classical version $G\leq \mathbb{G}$ (see below), the existence of a strictly intermediate quasi-subgroup $G\subsetneq \mathbb{S}\subsetneq \mathbb{G}$ does not imply a strictly intermediate quantum subgroup. For example, the finite dual $\widehat{A_5}$ has trivial classical version, and for any non-trivial subgroup $H\leq A_5$ the non-Haar idempotent $\mathds{1}_{H}$ gives a strict intermediate quasi-subgroup:
$$\{\varepsilon\}\subsetneq \mathbb{S}_H\subsetneq \widehat{A_5}.$$
However $\widehat{A_5}$ has no non-trivial quantum subgroups because $A_5$ is simple.

\bigskip

The idea for an example of an exotic intermediate quasi-subgroup would be to find a Pal set given by some condition that is satisfied by the `elements of $S_N$ in $S_N^+$' --- and some states non-zero on a commutator $[f,g]\in C(S_N^+)$ --- but not by the Haar state on $C(S_N^+)$. It will be seen that the `elements of $S_N$ in $S_N^+$' correspond to the characters on $C(S_N^+)$.
\subsection{The classical version of a quantum permutation group}
 The quotient of $C(\mathbb{G})$ by the commutator ideal is the algebra of functions on the characters on $C(\mathbb{G})$. The characters form a group $G$, with the group law given by the convolution:
$$\varphi_1\star \varphi_2=(\varphi_1\otimes\varphi_2)\Delta,$$
the identity is the counit, and the inverse is the reverse $\varphi^{-1}=\varphi\circ S$.

\bigskip

  As before the Birkhoff slice aids the analysis. See \cite{mcc} for more, where the following proof is sketched.
\begin{proposition}\label{character}
A state $\varphi$ on $C(\mathbb{G})$ is a character if and only if $\Phi(\varphi)$ is a permutation matrix.
\end{proposition}
\begin{proof}
If $\varphi$ is a character,
$$\varphi(u_{ij})=\varphi(u_{ij}^2)=\varphi(u_{ij})^2\Rightarrow \varphi(u_{ij})=0\text{ or 1}.$$
As it is doubly stochastic, it follows that $\Phi(\varphi)$ is a permutation matrix. Suppose now that $\Phi(\varphi)=P_\sigma$, the permutation matrix of $\sigma\in S_N$. Consider the GNS representation $(\mathsf{H}_{\sigma},\pi_{\sigma},\xi_{\sigma})$ associated to $\varphi$. By assumption
\begin{equation}\varphi(u_{ij})=\langle \xi_{\sigma},\pi_{\sigma}(u_{ij})(\xi_{\sigma})\rangle=\langle \pi_{\sigma}(u_{ij})(\xi_{\sigma}),\pi_{\sigma}(u_{ij})(\xi_{\sigma})\rangle=\|\pi_{\sigma}(u_{ij})(\xi_\sigma)\|^2= 0 \text{ or }1.\label{span}\end{equation}
For $f\in C(\mathbb{G})$, let $(f^{(n)})_{n\geq 1}\subset \mathcal{O}(\mathbb{G})$ converge to $f$. For each $f^{(n)}$, (\ref{span}) implies there exists $a_n\in \mathbb{C}$ such that
$$\pi_\sigma(f^{(n)})(\xi_\sigma)=a_n\xi_\sigma.$$
The representation $\pi_\sigma$ is norm continuous, and so $\pi_\sigma(f^{(n)})\to \pi_\sigma(f)$, and $(\pi_\sigma(f^{(n)}))_{n\geq 1}$ is Cauchy:
\begin{align*}\|\pi_\sigma(f^{(m)})-\pi_\sigma(f^{(n)})\|&\to 0
\\ \implies |a_m-a_n|\|\xi_\sigma\|&\to 0,
\end{align*}
which implies that $(a_n)_{n\geq 1}$ converges, to say $a_f\in\mathbb{C}$. The norm convergence of $f^{(n)}\to f$ implies the strong convergence of $\pi_\sigma(f^{(n)})$ to $\pi_\sigma(f)$:
\begin{align*}
  \pi_\sigma(f)\xi_\sigma & =\lim_{n\rightarrow \infty}\left(\pi_\sigma(f^{(n)})\xi_\sigma\right)  =\lim_{n\rightarrow \infty}(a_n\xi_\sigma)=a_f\xi_\sigma.
\end{align*}
Therefore
\begin{align*}
  \varphi(gf) & =\langle\xi_{\sigma},\pi_{\sigma}(gf)\xi_{\sigma}\rangle=\langle \xi_{\sigma},\pi_{\sigma}(g)\pi_{\sigma}(f)(\xi_{\sigma})\rangle
   \\&=\langle \xi_{\sigma},\pi_{\sigma}(g)a_f \xi_{\sigma}\rangle=a_f\langle\xi_{\sigma},\pi_{\sigma}(g)\xi_{\sigma}\rangle=\varphi(g)\varphi(f).
   \end{align*}
\end{proof}
Where $\pi_{\text{ab}}$ is the abelianisation, define $\operatorname{ev}_\sigma:C(\mathbb{G})\to \mathbb{C}$:
$$\operatorname{ev}_\sigma(f):=\pi_{\text{ab}}(f)(\sigma)\qquad (f\in C(\mathbb{G})).$$
This is a $\ast$-homomorphism, but in general  $\operatorname{ev}_\sigma$ need not be non-zero.
\begin{proposition}\label{charactertwo}
If $\varphi$ is a state on $C(\mathbb{G})$ such that $\Phi(\varphi)=P_\sigma$, then $\varphi=\operatorname{ev}_{\sigma}$.
\end{proposition}
\begin{proof}
Suppose that $\Phi(\varphi)=P_\sigma$. We know that $\operatorname{ev}_{\sigma}$ is a $\ast$-homomorphism, and by Proposition \ref{character} so is $\varphi$. As $C(\mathbb{G})$ admits a character, $\pi_{\text{ab}}$ is non-zero. Furthermore, as $\ast$-homomorphisms they are determined by their values on the generators:
$$\varphi(u_{ij})=\Phi(\varphi)_{ij}=\sigma_{ij}=\delta_{i,\sigma(j)}=\mathds{1}_{j\rightarrow i}(\sigma)=\pi_{\text{ab}}(u_{ij})(\sigma)=\operatorname{ev}_\sigma(u_{ij}).$$
\end{proof}
The \emph{classical version} of $\mathbb{G}$ is therefore the finite group $G\leq S_N$  given by:
$$G:=\{\operatorname{ev}_\sigma\,\colon\,\sigma\in S_N,\,\operatorname{ev}_\sigma\neq 0\}.$$

References to $u_{ij}$ in the below are in the embedding:
$$C(\mathbb{G})\subseteq C(\mathbb{G})^{**}.$$
\begin{proposition}\label{charsupport}
Associated to each character $\operatorname{ev}_\sigma$ on $C(\mathbb{G})$    is a \emph{support projection} $p_\sigma\in C(\mathbb{G})^{**}$ such that:
\begin{enumerate}
  \item[(i)] $p_{\sigma}$ is a central projection in $C(\mathbb{G})^{**}$, and $p_{\sigma}p_{\tau}=\delta_{\sigma,\tau}p_\sigma$.
  \item[(ii)]   $p_{\sigma}=u_{\sigma(1),1}\wedge u_{\sigma(2),2}\wedge\ldots\wedge u_{\sigma(N),N}.$

\end{enumerate}\end{proposition}

\begin{proof}
  \begin{enumerate}

    \item[(i)]  Follows from the fact the $p_\sigma$ are support projections of characters.

    \item[(ii)]  Let $$q_\sigma=u_{\sigma(1),1}\wedge u_{\sigma(2),2}\wedge\ldots\wedge u_{\sigma(N),N}.$$  Define $$f_\sigma:=u_{\sigma(1),1}\cdots u_{\sigma(N),N}.$$ The sequence $(f_{\sigma}^n)_{n\geq 1}\subset C(\mathbb{G})$ converges $\sigma$-weakly to $q_\sigma$. The extension $\omega_\sigma$ of $\operatorname{ev}_\sigma$ is a character implying that:
    \begin{align*}
      \omega_\sigma(q_\sigma) & =\lim_{n\rightarrow \infty}\omega_\sigma(f_{\sigma}^n) =1 \implies p_\sigma\leq q_\sigma.
    \end{align*}
    Suppose $r:=q_{\sigma}-p_\sigma$ is non-zero. Then there exists a state $\omega_r$ on $C(\mathbb{G})^{**}$ such that $\omega_r(r)=1$.
    Define a state $\varphi_r$ on $C(\mathbb{G})$ by:
    $$\varphi_r(f)=\omega_r(rfr)\qquad (f\in C(\mathbb{G})).$$
    Then $\varphi_r(u_{\sigma(j),j})=1\implies \varphi_r=\operatorname{ev}_\sigma$,
 by Proposition \ref{charactertwo}, with equal extensions $\omega_r$ and $\omega_\sigma$. However, in this case
 $$\omega_\sigma(p_\sigma)=\omega_r(p_\sigma)=0,$$
 and this contradiction gives $q_\sigma=p_\sigma$.

     \end{enumerate}
\end{proof}
 In the following, whenever $\operatorname{ev}_{\sigma}=0$, then so is $p_{\sigma}$. Properties of the bidual summarised in Section \ref{bid} are used.

\begin{theorem}\label{glp}
Where $G\leq \mathbb{G}$ is the classical version, define
$$p_G:=\sum_{\sigma\in G}p_{\sigma}.$$
Then $p_{G}$ is a group-like projection in $C(\mathbb{G})^{**}$. In addition, $p_G$ is the support projection of the Haar idempotent $h_{C(G)}\circ \pi_{\text{ab}}$.
\end{theorem}
\begin{proof}
Note $p_G$ is non-zero, as $p_{\varepsilon}p_G=p_{\varepsilon}$.
 Consider $p_{\sigma}\neq 0$. Let $(p^{\lambda}_\sigma)\subset \mathcal{O}(\mathbb{G})$  converge $\sigma$-weakly to $p_{\sigma}\in C(\mathbb{G})^{**}$. The extension of $\Delta$ is $\sigma$-weakly continuous, and recall that $p_\sigma$ is a meet of projections in $\mathcal{O}(\mathbb{G})$:
\begin{align*}
  \Delta^{**}(p_\sigma) & =\Delta^{**}(u_{\sigma(1),1}\wedge u_{\sigma(2),2}\wedge\cdots\wedge u_{\sigma(N),N}) \\
   & =\Delta(u_{\sigma(1),1})\wedge \Delta(u_{\sigma(2),2})\wedge\cdots\wedge\Delta(u_{\sigma(N),N}) \\
   & =\lim_{n\to \infty}\left[\Delta(u_{\sigma(1),1}) \Delta(u_{\sigma(2),2})\cdots\Delta(u_{\sigma(N),N})\right]^n
\end{align*}
Consider, for $p_\tau\neq 0$
\begin{align*}
&  \Delta(u_{\sigma(1),1}) \Delta(u_{\sigma(2),2})\cdots\Delta(u_{\sigma(N),N})(\mathds{1}_{\mathbb{G}}\otimes p_\tau) \\&=\left(\sum_{k_1,\dots,k_N=1}^Nu_{\sigma(1),k_1}u_{\sigma(2),k_2}\cdots u_{\sigma(N),k_N}\otimes u_{k_1,1}u_{k_2,2}\cdots u_{k_N,N}\right)(\mathds{1}_{\mathbb{G}}\otimes p_\tau)
\end{align*}
Note $p_\tau$ is central and
$$p_\tau u_{kj}=\begin{cases}
                  p_\tau u_{kj}, & \mbox{if } k=\tau(j) \\
                  0, & \mbox{otherwise}
                \end{cases},$$
                and so
                \begin{align*}
                &\Delta(u_{\sigma(1),1}) \Delta(u_{\sigma(2),2})\cdots\Delta(u_{\sigma(N),N})(\mathds{1}_{\mathbb{G}}\otimes p_\tau)
                \\&=u_{\sigma(1),\tau(1)}u_{\sigma(2),\tau(2)}\cdots u_{\sigma(N),\tau(N)}\otimes u_{\tau(1),1}u_{\tau(2),2}\cdots u_{\tau(N),N}p_{\tau}
                \\&=(u_{\sigma(1),\tau(1)}u_{\sigma(2),\tau(2)}\cdots u_{\sigma(N),\tau(N)}\otimes u_{\tau(1),1}u_{\tau(2),2}\cdots u_{\tau(N),N})(\mathds{1}_{\mathbb{G}}\otimes p_{\tau})
                \end{align*}
                Now
\begin{align*}
\Delta^{**}(p_\sigma)(\mathds{1}_{\mathbb{G}}\otimes p_\tau)&=\lim_{n\to\infty}\left[\Delta(u_{\sigma(1),1}) \Delta(u_{\sigma(2),2})\cdots\Delta(u_{\sigma(N),N})\right]^n(\mathds{1}_{\mathbb{G}}\otimes p_\tau)
\\ &=\lim_{n\to\infty}\left[\left(\Delta(u_{\sigma(1),1}) \Delta(u_{\sigma(2),2})\cdots\Delta(u_{\sigma(N),N})\right)^n(\mathds{1}_{\mathbb{G}}\otimes p_\tau)\right]
\\ &=\lim_{n\to\infty}\left[\left(\Delta(u_{\sigma(1),1}) \Delta(u_{\sigma(2),2})\cdots\Delta(u_{\sigma(N),N})\right)^n(\mathds{1}_{\mathbb{G}}\otimes p_\tau)^n\right]
\\&=\lim_{n\to\infty}\left[\left(\Delta(u_{\sigma(1),1}) \Delta(u_{\sigma(2),2})\cdots\Delta(u_{\sigma(N),N})\right)(\mathds{1}_{\mathbb{G}}\otimes p_\tau)\right]^n
\\&=\lim_{n\to\infty}\left[(u_{\sigma(1),\tau(1)}u_{\sigma(2),\tau(2)}\cdots u_{\sigma(N),\tau(N)}\otimes u_{\tau(1),1}u_{\tau(2),2}\cdots u_{\tau(N),N})(\mathds{1}_{\mathbb{G}}\otimes p_{\tau})\right]^n
\\&=\lim_{n\to\infty}\left[(u_{\sigma(1),\tau(1)}u_{\sigma(2),\tau(2)}\cdots u_{\sigma(N),\tau(N)}\otimes u_{\tau(1),1}u_{\tau(2),2}\cdots u_{\tau(N),N})^n(\mathds{1}_{\mathbb{G}}\otimes p_{\tau})^n\right]
\\&=\lim_{n\to\infty}\left[(u_{\sigma(1),\tau(1)}u_{\sigma(2),\tau(2)}\cdots u_{\sigma(N),\tau(N)}\otimes u_{\tau(1),1}u_{\tau(2),2}\cdots u_{\tau(N),N})^n(\mathds{1}_{\mathbb{G}}\otimes p_{\tau})\right]
\\&=\lim_{n\to\infty}\left[u_{\sigma(1),\tau(1)}u_{\sigma(2),\tau(2)}\cdots u_{\sigma(N),\tau(N)}\otimes u_{\tau(1),1}u_{\tau(2),2}\cdots u_{\tau(N),N})^n\right](\mathds{1}_{\mathbb{G}}\otimes p_{\tau})
\\&=(p_{\sigma\tau^{-1}}\otimes p_{\tau})(\mathds{1}_{\mathbb{G}}\otimes p_{\tau})=p_{\sigma\tau^{-1}}\otimes p_{\tau},
\end{align*}
the last equality following from the fact that if bounded nets $(f_i),\,(g_i)$ in von Neumann algebras converge to $f,\,g$ $\sigma$-weakly, then the net $(f_i\otimes g_i)$ converges $\sigma$-weakly to $f\otimes g$ in the von Neumann tensor product. Finally, sum $\Delta^{**}(p_\sigma)(\mathds{1}_{\mathbb{G}}\otimes p_\tau)$ over $\sigma,\,\tau\in G$.

\bigskip

Note that $C(G)=C(\mathbb{G})/N_{\text{ab}}$ is finite dimensional, and so by (\ref{directsum}):
$$C(\mathbb{G})^{**}\cong C(G)\oplus N_{\text{ab}}^{**}.$$
It follows that the support projection of $h_{C(G)}\circ \pi_{\text{ab}}$ is $p_G$.
\end{proof}

\subsection{The (classically) random and truly quantum parts of a quantum permutation}
In the case of $C(S_N^+)$, define $p_C:=p_{S_N}$ and $p_Q:=\mathds{1}_{S_N^+}-p_C$. Recall:
$$\varphi\in S_N^+ \text{ is a quantum permutation}\iff\varphi\text{ a state on }C(S_N^+).$$
\begin{definition}
Let $\varphi\in S_N^+$ be a quantum permutation. Say that $\varphi$
\begin{enumerate}
  \item[(i)] is a \emph{(classically) random permutation} if $\omega_\varphi(p_Q)=0$,
  \item[(ii)] is a \emph{genuinely quantum permutation} if $\omega_\varphi(p_Q)>0$,
  \item[(iii)] is a \emph{mixed quantum permutation} if  $0<\omega_\varphi(p_Q)<1$,
  \item[(iv)] is a \emph{truly quantum permutation} if $\omega_\varphi(p_Q)=1$.
\end{enumerate}
\end{definition}
Random permutations are in bijection with probability measures $\nu\in M_p(S_N)$:
\begin{align*}
\varphi\text{ random }&\iff \varphi=\varphi_\nu\text{ where }\\
\varphi_\nu(f)&:=\sum_{\sigma\in S_N}\pi_{\text{ab}}(f)(\sigma)\nu(\{\sigma\})\qquad (f\in C(S_N^+)).
\end{align*}

\begin{theorem}\label{absorb}
Suppose $h_{S_N}$ is the state on $C(S_N^+)$ defined by $h_{C(S_N)}\circ \pi_{\text{ab}}$. Then if
$$\varphi\star h_{S_N}=h_{S_N}=h_{S_N}\star \varphi,$$
$\varphi$ is a random permutation.
\end{theorem}
\begin{proof}
This follows from Theorem \ref{pqprops}.

\end{proof}

\subsection{Exotic quasi-subgroups}
\begin{theorem}\label{idemtotty}
Let $\varphi\in S_N^+$ be genuinely quantum, $\omega_{\varphi}(p_Q)>0$, and $h_{S_N}\in S_N^+$ the Haar idempotent $h_{C(S_N)}\circ \pi_{\text{ab}}$. Form the idempotent $\phi_\varphi$ from the weak-* limit of Ces\`{a}ro means of $\varphi$, and then define an idempotent:
\begin{equation}\phi:=w^*\text{-}\lim_{n\to\infty}\frac{1}{n}\sum_{k=1}^n(h_{S_N}\star \phi_\varphi)^{\star k}.\label{idempot}\end{equation}
Then the quasi-subgroup generated satisfies:
 $$S_N\subsetneq \mathbb{S}_{\phi}\subseteq S_N^+.$$
 \end{theorem}
 \begin{proof}
For any $\sigma\in S_N$, and $\phi_n$ a Ces\`{a}ro mean of $(h_{S_N}\star \phi_\varphi)$, the weak-* continuity of convolution operators gives:
$$\operatorname{ev}_\sigma\star \phi_n=\phi_n\implies w^*\text{-}\lim_{n\to\infty}(\operatorname{ev}_\sigma\star\phi_n)=\phi\implies \operatorname{ev}_\sigma\star \phi=\phi\implies \phi\star\operatorname{ev}_{\sigma^{-1}}=\phi.$$
by Proposition \ref{reverse}. Similarly $\operatorname{ev}_{\sigma^{-1}}\star \phi_n\to \phi$ which implies that $\phi\star\operatorname{ev}_{\sigma}=\phi$, and so $S_N\subseteq \mathbb{S}_{\phi}$

\bigskip

Now suppose for the sake of contradiction that $\phi$ is random. Then
$$\phi\star h_{S_N}=h_{S_N}=h_{S_N}\star \phi.$$
However for all Ces\`{a}ro means $\phi_n$:
$$\phi_n\star \varphi=\phi_n\implies \phi\star \varphi=\phi\implies h_{S_N}\star \varphi=h_{S_N},$$
by left convolving both sides of  $ \phi\star \varphi=\phi$ with $h_{S_N}$. To apply Theorem \ref{absorb}, and conclude that $\varphi$ is random, it must also be shown that $\varphi\star h_{S_N}=h_{S_N}$. Note that inductively Proposition \ref{reverse} applies to convolutions of more than two states:
$$\phi_n\circ S=\frac{1}{n}\sum_{k=1}^n((\phi_\varphi\circ S)\star (h_{S_N}\circ S) )^{\star k}=\frac{1}{n}\sum_{k=1}^n(\phi_\varphi\star h_{S_N})^{\star k},$$
as $\phi_\varphi$ and $h_{S_N}$ are idempotent states. The weak*-limit of $\phi_n\circ S$ is $\phi\circ S=\phi$, and the preceding argument with a right convolution applied to $\phi_n\circ S$ yields  $\varphi\star h_{S_N}=h_{S_N}$.
 \end{proof}

If in fact for all genuinely quantum $\varphi\in S_N^+$ it is the case that $\mathbb{S}_{\phi}=S_N^+$ for $\phi$ given by (\ref{idempot}), then the maximality conjecture holds, and it is tenable to say that $h_{S_N}$ and \emph{any} genuinely quantum permutation $\varphi\in S_N^+$ generates  $S_N^+$.

\bigskip

\section{Convolution dynamics}
This section will explore, with respect to $p_Q\in C(S_N^+)^{**}$, the qualitative dynamics of quantum permutations under convolution.  The results of this section are  illustrated qualitatively in a phase diagram, Figure \ref{phase}. In fact the results of this section hold more generally. These results in fact hold for a classical subgroup $H\leq G$ under the condition that there exists $h_1,h_2,h_3,h_4\in H^c$ such that $h_1h_2\in H$ and $h_3h_4\in H^c$, with the  projection $\mathds{1}_{H^c}\in C(G)$ playing the role of $p_Q$.  They also hold more generally for other quantum subgroups $\mathbb{G}\leq S_N^+$  such that $h_{\mathbb{G}}:=h_{C(\mathbb{G})}\circ \pi_{C(\mathbb{G})}$ has group-like support projection $p_{\mathbb{G}}\in C(S_N^+)^{**}$, and there exists quantum permutations $\varphi_1,\,\varphi_2,\,\varphi_3,\,\varphi_4\in S_N^+$ supported  ``off'' $\mathbb{G}$  in the sense that $\omega_{\varphi_i}(p_{\mathbb{G}})=0$ such that:
  \begin{enumerate}
    \item[(i)] $\omega_{\varphi_1\star\varphi_2}(p_{\mathbb{G}})=0$,
    \item[(ii)] $\omega_{\varphi_3\star\varphi_3}(p_{\mathbb{G}})=1$.
  \end{enumerate}

\subsection{The convolution of random and truly quantum permutations}

\begin{lemma}\label{qlemma}
Suppose $p\in C(\mathbb{G})^{**}$ is a group-like projection. Then, where $q:=\mathds{1}_{\mathbb{G}}-p$:
$$\Delta^{**}(q)(q\otimes p)=q\otimes p.$$
\end{lemma}
\begin{proof}
  Expand
  $$\Delta^{**}(p+q)(\mathds{1}_{\mathbb{G}}\otimes p)=(\mathds{1}_{\mathbb{G}}\otimes p),$$
  then multiply on the right with $q\otimes p$.
\end{proof}
\begin{proposition}\label{rules}
\begin{enumerate}
  \item[(i)] The convolution of random permutations is random.
  \item[(ii)] The convolution of a truly quantum permutation and a random permutation is truly quantum.
  \item[(iii)] The convolution of a truly quantum permutations can be random, mixed, or truly quantum.
\end{enumerate}
\end{proposition}
\begin{proof}
  \begin{enumerate}
    \item[(i)] This is straightforward.
    \item[(ii)] Let $\varphi$ be truly quantum, and $\varphi_\nu$ random with extension $\omega_\nu$. Let $(p_Q^\lambda)\subset \mathcal{O}(S_N^+)$ converge $\sigma$-weakly to $p_Q$. Recalling the group-like projection $p_C=\sum_{\sigma\in S_N}\operatorname{ev}_\sigma$,  using Lemma \ref{qlemma}, mimic the proof of Theorem \ref{pqprops}, hitting both sides of
        $$\Delta^{**}(p_Q)(p_Q\otimes p_C)=p_Q\otimes p_C,$$
        with $\omega_\varphi\otimes \omega_\nu$, to yield:
        $$\omega_{\varphi\star \varphi_\nu}(p_Q)=1,$$
        i.e. $\varphi\star \varphi_\nu$ is truly quantum.
    \item[(iii)] It will be seen in Corollary \ref{Haartq} that the Haar state is truly quantum. Note that for any $N\geq 4$, the Kac--Paljutkin quantum group can be embedded $G_0\leq S_N^+$ via $\pi_{G_0}$. It can be shown that   $E^{11}\circ \pi_{G_0}$ is truly quantum, and  $(E^{11}\circ \pi_{G_0})^{\star 2}=\varphi_\nu$ is a random permutation ((4.6), \cite{mcc}). Let $0\leq c\leq 1$ and consider the truly quantum permutation:
$$\varphi:=\sqrt{1-c}\,(E^{11}\circ \pi_{G_0})+(1-\sqrt{1-c})\,h.$$
Then:
$$\varphi^{\star 2}=(1-c)\varphi_\nu+c \,h\implies \omega_{\varphi^{\star 2}}(p_Q)=c.$$
  \end{enumerate}
\end{proof}

\begin{proposition}
A quantum permutation $\varphi\in S_N^+$ can be written as a convex combination of a random permutation and a truly quantum permutation.
\end{proposition}
\begin{proof}
If $\varphi$ is random, or truly quantum, the result holds. Assume $\varphi$ is mixed. The projections $p_C,\,p_Q\in C(S_N^+)^{**}$ are central, and thus
\begin{align*}
  \varphi =\omega_\varphi(p_C)\,\widetilde{p_C}\varphi+\omega_\varphi(p_Q)\,\widetilde{p_Q}\varphi,
\end{align*}
with $\widetilde{p_C}\varphi$ random, and $\widetilde{p_Q}\varphi$  truly quantum.
\end{proof}
\begin{corollary}\label{makeran}
  If the convolution of two quantum permutations is a random permutation, then either both are random, or both are truly quantum.
\end{corollary}
\begin{definition}
Let $\varphi\in S_N^+$ be a quantum permutation. Define $\varphi_C:=\widetilde{p_C}\varphi$, the (classically) random part of $\varphi$, and $\varphi_Q:=\widetilde{p_Q}\varphi$, the truly quantum part of $\varphi$. A quantum permutation $\varphi\in S_N^+$ is called $\alpha$-quantum if $\omega_\varphi(p_Q)=\alpha$.
\end{definition}

\begin{proposition}
If $\varphi\in S_N^+$ is a mixed quantum permutation with $0<\omega_\varphi(p_Q)<1$, then no finite convolution power $\varphi^{\star k}$ is random, or truly quantum.
\end{proposition}
\begin{proof}
Let $\alpha:=\omega_\varphi(p_Q)$ and write $\varphi=(1-\alpha)\varphi_C+\alpha\,\varphi_Q$:
$$\varphi^{\star k}>(1-\alpha)^k\varphi_C^{\star k}\implies \omega_{\varphi^{\star k}}(p_Q)\leq  1-(1-\alpha)^k,$$
so no $\varphi^{\star k}$ is truly quantum. In addition, $\varphi^{\star k}=\varphi\star \varphi^{\star(k-1)}$ cannot be random, by Corollary \ref{makeran}, because $\varphi$ is neither random nor truly quantum.

\end{proof}

\begin{proposition}\label{alphabeta}
If $\varphi\in S_N^+$ is $\alpha$-quantum and $\rho\in S_N^+$ is $\beta$-quantum, then
$$\alpha+\beta-2\alpha\beta\leq \omega_{\varphi\star \rho}(p_Q)\leq \alpha+\beta-\alpha\beta.$$
\end{proposition}
\begin{proof}
Note that $\varphi\star \rho$ equals:
$$(1-\alpha)(1-\beta) (\varphi_C\star \rho_C)+\beta (1-\alpha)(\varphi_C\star \rho_Q)+\alpha(1-\beta)(\varphi_Q\star \rho_C)+\alpha\beta(\varphi_Q\star \rho_Q).$$
Now apply Proposition \ref{rules}.
\end{proof}
\begin{definition}
  Where $\overline{(\varphi,\rho)}=(\varphi+\rho)/2$ is the mean of two quantum permutations, a \emph{quantum strictly 1-increasing} pair of quantum permutations $\varphi_1,\varphi_2\in S_N^+$ are a pair such that:
  $$\omega_{\varphi_1\star\varphi_2}(p_Q)>\omega_{\overline{(\varphi_1,\varphi_2)}}(p_Q).$$
  A \emph{quantum strictly 2-increasing} pair of quantum permutations are a pair such that:
  $$\omega_{(\varphi_1\star\varphi_2)^{\star 2}}(p_Q)>\omega_{\varphi_1\star\varphi_2}(p_Q)>\omega_{\overline{(\varphi_1,\varphi_2)}}(p_Q).$$
  Inductively, a  \emph{quantum strictly $(n+1)$-increasing} pair of quantum permutations are a pair such that:
  $$\omega_{(\varphi_1\star\varphi_2)^{\star (2^n)}}(p_Q)>\omega_{(\varphi_1\star\varphi_2)^{\star \left(2^{n-1}\right)}}(p_Q)>\cdots>\omega_{\varphi_1\star\varphi_2}(p_Q)>\omega_{\overline{(\varphi_1,\varphi_2)}}(p_Q).$$

\end{definition}

\begin{proposition}\label{region}
Let $\varphi_1\in S_N^+$ be $\alpha$-quantum, and $\varphi_2\in S_N^+$ be $\beta$-quantum.
\begin{enumerate}
  \item[(i)] If $(\alpha,\beta)\neq (0,0)$, then if $\alpha=1/4$ or $\beta< \alpha/(4\alpha-1)$,  the pair $(\varphi_1,\varphi_2)$ is \emph{quantum strictly 1-increasing}.
  \item[(ii)] If $(\alpha,\beta)\neq (0,0)$, and $\beta=\alpha/(4\alpha-1)$, then:
  $$\omega_{\varphi_1\star \varphi_2}(p_Q)\geq\omega_{\overline{(\varphi_1,\varphi_2)}}(p_Q).$$
  Equality is possible, with e.g. quantum permutations coming from the Kac--Paljutkin quantum group $G_0\leq S_N^+$.
  \item[(iii)] If  $\beta> \alpha/(4\alpha-1)$ then $\omega_{\varphi_1\star \varphi_2}(p_Q)$ can be less than, equal to, or greater than $\omega_{\overline{(\varphi_1,\varphi_2)}}(p_Q)$.
  \item[(iv)] Let
  $$(S_{N}^+\times S_N^+)_{\alpha,\beta}:=\{(\varphi,\rho)\,\colon\,\omega_\varphi(p_Q)=\alpha,\,\omega_\rho(p_Q)=\beta\}.$$
  Then
$$\max\{|\omega_{\varphi_1\star \varphi_2}(p_Q)-\omega_{\varphi_3\star \varphi_4}(p_Q)|\,\colon\,(\varphi_1,\varphi_2), \,(\varphi_3,\varphi_4)\in (S_{N}^+\times S_N^+)_{\alpha,\beta}\}=\alpha\beta.$$

\end{enumerate}
\end{proposition}
\begin{proof}
For (i)-(iii) apply Proposition \ref{alphabeta}. For (iv), the maximum in Proposition \ref{alphabeta} is attained for
\begin{align*}
 \varphi_1&=(1-\alpha)\,h_{S_N}+\alpha\,h \\
  \varphi_2 & =(1-\beta)\,h_{S_N}+\beta\,h \\
 \varphi_3 & =(1-\alpha)\,h_{S_N}+\alpha\,(E^{11}\circ\pi_{G_0}) \\
  \varphi_4 & =(1-\beta)\,h_{S_N}+\beta\,(E^{11}\circ\pi_{G_0})
\end{align*}
\end{proof}
Suppose that $\varphi_1$ is $\alpha$-quantum, and $\varphi_2$ is $\beta$-quantum. The subset of $S_N^+\times S_N^+$ given by  condition (i) is called the $Q_I$-region. In this region the dynamics of the convolution $(\varphi_1,\varphi_2)\to \varphi$ with respect to $p_Q$ cannot be too wild:
$$\omega_{\varphi_1\star \varphi_2}(p_Q)\in\left(\omega_{\overline{(\varphi_1,\varphi_2)}}(p_Q),\omega_{\overline{(\varphi_1,\varphi_2)}}(p_Q)+\alpha\beta\right].$$
Note that the width of this interval tends to zero for $\alpha\beta\to 0$.

\bigskip

On the other hand, the region of $S_N^+\times S_N^+$ given by (iii) is called the $Q_W$-region, and the dynamics can be more wild here. Given an arbitrary pair of quantum permutations in this region, the convolution can be more, equal, or less quantum than the mean, and, as $\alpha\beta\to 1$, over the collection of  $(\varphi,\rho)\in Q_W$  the possible range of values of $\omega_{\varphi\star \rho}(p_Q)$ tends to one. Tracing from $Q_I$ towards $Q_W$, on the boundary $\partial_W$ (given by (ii)) `conservation of quantumness',
$$\omega_{\varphi_1\star \varphi_2}(p_Q)=\omega_{\overline{(\varphi_1,\varphi_2)}}(p_Q),$$
becomes possible for the first time.

\bigskip

Similarly, higher order regions can be defined:
\begin{enumerate}
  \item The region $Q_{2I}\subseteq Q_{I}$ given by $\beta<(2\alpha-1)/(2\alpha-2)$ consists of quantum strictly 2-increasing pairs;
  \item The region $Q_{3I}\subseteq Q_{2I}$ given by $\beta<1-\sqrt{2}/(1-2\alpha)$ consists of quantum strictly 3-increasing pairs;
  \item The region $Q_{\frac12 W}\subseteq Q_W$ given by $\beta>(1-1/\sqrt{2})/\alpha$ consists of pairs of quantum permutations $(\varphi_1,\varphi_2)$ such that the pair $(\varphi_1\star\varphi_2,\varphi_1\star \varphi_2)\not\in  Q_{2I}$, etc.
\end{enumerate}

\begin{figure}[ht]\begin{center} \includegraphics[width=0.55\textwidth]{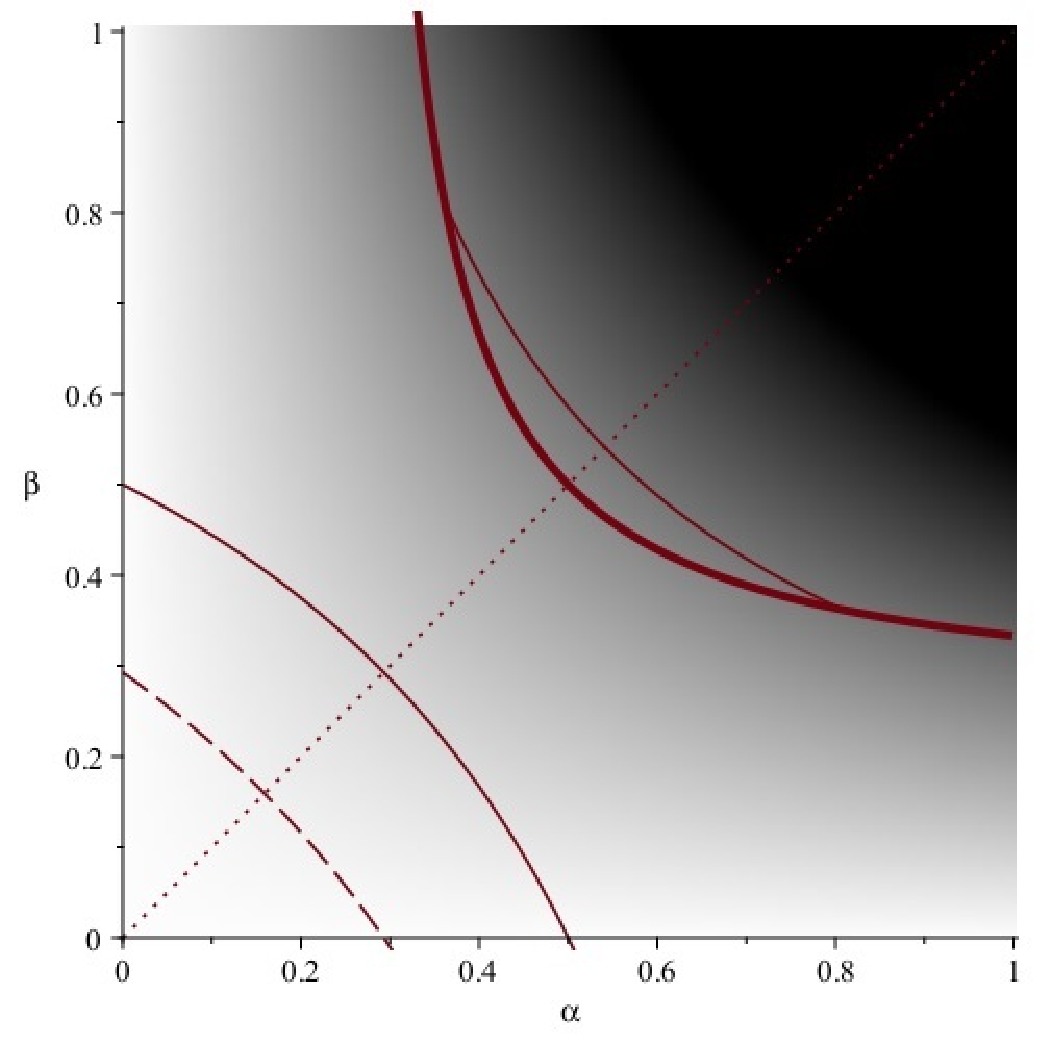}\caption{The phase diagram for the convolution of $\alpha$-quantum and $\beta$-quantum permutations. The phases are quantum increasing, $Q_I$, in the bottom left, and quantum wild, $Q_W$, in the top right, with the bold line $\partial_W$ the boundary. From the bottom left,  $Q_{3I}\subset Q_{2I}\subset Q_{I}$, and then touching $\partial_W$ on the diagonal, $Q_{\frac12W}\subset Q_{W}$. The region $Q_{\frac12 W}$ is such that the convolution of states from this region cannot be too close to random: indeed the convolution cannot fall inside $Q_{2I}$. The line $\alpha=\beta$ represents $(\varphi,\varphi)\to\varphi^{\star 2}$. The shading is proportional to $\alpha\beta$ (see Proposition \ref{region} (iv)).\label{phase}}\end{center}\end{figure}

\subsection{The truly quantum part of an idempotent state}
\begin{corollary}\label{idempot2}
If $\phi\in S_N^+$ is an idempotent state, then
$$\omega_\phi(p_Q)\in \{0\}\cup[1/2,1].$$
\end{corollary}
\begin{proof}
If $\phi$ is an idempotent state,
$$\omega_\phi(p_Q)=\omega_{\phi\star\phi}(p_Q).$$
The rest follows from Proposition \ref{region}.
\end{proof}
An idempotent on the boundary $\partial_{W}$ is the Haar idempotent $h_{G_0}$ associated with the Kac--Paljutkin quantum group $G_0\leq S_4^+$ which satisfies $\omega_{h_{G_0}}(p_Q)=1/2$.

\begin{example}\label{limit}
Let $\mathbb{G}$ be a finite quantum group given by $\pi:C(S_N^+)\to C(\mathbb{G})$. Where $G\leq \mathbb{G}$ is the classical version, the $\sigma$-weak extension $\pi^{**}$ to the biduals maps onto $C(\mathbb{G})$, and in particular $\pi^{**}(p_\sigma)\in C(\mathbb{G})$ is the support projection of
$$f\mapsto \pi_{\text{ab}}(\pi(f))(\sigma)\qquad(f\in C(S_N^+)).$$
Let $h_{\mathbb{G}}:=h_{C(\mathbb{G})}\circ \pi$ with extension to the bidual $\omega_{\mathbb{G}}$. From e.g. \cite{kpa}:
$$\omega_{\mathbb{G}}(p_\sigma)=\frac{1}{\dim C(\mathbb{G})}\qquad (\sigma\in G).$$
This implies that
\begin{equation}\omega_{\mathbb{G}}(p_Q)= 1-\frac{|G|}{\dim C(\mathbb{G})}.\label{finitech}\end{equation} Let $N\geq 9$, where $S_N$ is generated by elements $\sigma,\tau$ of order two and three \cite{mil}, and thus there is an embedding $\widehat{S_N}\leq S_5^+$ given by magic unitaries $u^\sigma\in M_2(C(\widehat{S_N}))$ and $u^\tau\in M_3(C(\widehat{S_N}))$ (Chapter 13, \cite{ba1}):
$$u=\begin{bmatrix}
      u^{\sigma} & 0 \\
      0 & u^{\tau}
    \end{bmatrix}.$$
A finite dual $\widehat{\Gamma}\leq S_N^+$ has classical version with order equal to the number of one dimensional representations of $\Gamma$ (see \cite{mcc} for more). Therefore the classical version of $\widehat{S_N}$ is $\mathbb{Z}_2$ and so, for $N\geq 9$, the associated Haar idempotent:
\begin{equation}\omega_{\widehat{S_N}}(p_Q)=1-\frac{2}{N!},\label{limpt}\end{equation}
which tends to one for $N\to \infty$.
\end{example}

   This suggests the following study:  consider
$$\chi_N:=\{\omega_\phi(p_Q)\,\colon\,\phi\in S_N^+,\,\phi\star\phi=\phi\}.$$
It is the case that $\chi_N=\{0\}$ for $N\leq 3$, and otherwise a non-singleton. By (\ref{limpt}), $1$ is a limit point for $\chi_5\cap[1/2,1)$. Is there any other interesting behaviour: either at fixed $N$, or asymptotically $N\to\infty$?

\bigskip

It seems unlikely that there exists a exotic finite quantum permutation group, but something can be said:
\begin{proposition}
An exotic finite quantum permutation group at order $N$ satisfies:
$$\dim C(\mathbb{G})\geq 2N!$$
In particular, there is no exotic finite quantum group with $\dim C(\mathbb{G})<1440$.
\end{proposition}
\begin{proof}
This follows from (\ref{finitech}) and Corollary \ref{idempot2}, and the fact that any exotic quantum permutation group $S_N\lneq\mathbb{G}\lneq S_N^+$ must satisfy $N\geq 6$.
\end{proof}
\subsection{Periodicity}
A periodicity in convolution powers of random permutations is possible. For example, suppose that $G\leq S_N$ and $N\lhd G$ is a normal subgroup. Consider the probability  $\nu$ uniform on the coset $Ng$. Then, where $\varphi_\nu\in S_N^+$ is the associated state:
$$\varphi_\nu(f)=\sum_{\sigma \in S_N}\pi_{\text{ab}}(f)(\sigma)\nu(\{\sigma\})=\frac{1}{|Ng|}\sum_{\tau\in N}\pi_{\text{ab}}(f)(\tau g)\qquad (f\in C(S_N^+)),$$
the convolution powers $(\varphi_\nu^{\star k})_{k\geq 0}$ are periodic, with period equal to the order of $g$.

\bigskip

There can also be periodicity with respect to $p_Q$. For example, $\varphi:=E^{11}\circ \pi_{G_0}$ is such that
$$\omega_{\varphi^{\star k}}(p_Q)=\begin{cases}
                         0, & \mbox{if } k\text{ even}, \\
                         1, & \mbox{if } k\text{ odd}.
                       \end{cases}$$

\begin{proposition}
Suppose that $\varphi\in S_N^+$ is truly quantum. If $\varphi^{\star k}$ is  random, then $\varphi^{\star (k+1)}$ is truly quantum.
\end{proposition}
\begin{proof}
Follows from Corollary \ref{makeran}.
\end{proof}
\begin{corollary}
Suppose that a truly quantum permutation $\varphi$ has a random finite convolution power. Let $k_0$ be the smallest such power. Then:
$$\omega_{\varphi^{k}}(p_Q)=\begin{cases}
                              0, & \mbox{if } k\mod k_0=0, \\
                              1, & \mbox{otherwise}.
                            \end{cases}$$
\end{corollary}
Is there a quantum permutation with $k_0>2$? This phenomenon suggests looking at when the classical version of $\mathbb{G}$ is a normal quantum subgroup $G\lhd \mathbb{G}$. However, in general, the classical periodicity associated with probability measures constant on cosets of $N\lhd G$ for $G\leq S_N$ does not extend to the quantum case. See \cite{mc1}, Section 4.3.1.

\section{Integer fixed points quantum permutations}
For a quantum permutation group $\mathbb{G}$,  consider the `number of fixed points':
$$\operatorname{fix}:=\sum_{j=1}^N u_{jj}.$$
This is the trace of the fundamental representation, the so-called `main character'. In the classical case, with fundamental representation $(\mathds{1}_{j\to i})_{i,j=1}^N$ of $C(S_N)$, the  main character counts the number of fixed points of a permutation.
Note that the spectrum $\sigma(\operatorname{fix})\subseteq [0,N]$. Consider a finite partition $\mathcal{P}$ of the spectrum into Borel subsets,
$$\sigma(\operatorname{fix})=\bigsqcup_{i=1}^m E_i.$$
Borel functional calculus can be used to attach a (pairwise-distinct) label $\lambda_i$ to each $E_i\subseteq\sigma(\operatorname{fix})$, and the number of fixed points of a quantum permutation $\varphi$ can be measured using $\operatorname{fix}_{\mathcal{P}}\in C(\mathbb{G})^{**}$ given by:
$$\operatorname{fix}_{\mathcal{P}}:=\sum_{i=1}^m \lambda_i\,\mathds{1}_{E_i}(\operatorname{fix}).$$
Measurement  is in the sense of algebraic quantum probability and the Gelfand--Birkhoff picture: when a quantum permutation $\varphi\in \mathbb{G}$ is measured with a finite spectrum observable $f=\sum_{\lambda \in \sigma(f)}\lambda \, p_\lambda$ in the bidual $C(\mathbb{G})^{**}$, the result is an element of $\sigma(f)$, with $f=\lambda$ with probability $\omega_{\varphi}(p_\lambda)$, and in that event there is wave-function collapse to $\widetilde{p_\lambda}\varphi$.
\begin{definition}
A quantum permutation $\varphi\in S_N^+$ has \emph{integer fixed points only} if for all Borel subsets $E\subseteq \sigma(\operatorname{fix})$,
$$E\cap \{0,1,\dots,N\}=\emptyset \implies \omega_{\varphi}(\mathds{1}_{E}(\operatorname{fix}))=0.$$
Equivalently, if
$$\omega_{\varphi}(\mathds{1}_{\{0,1,\dots,N\}}(\operatorname{fix}))=1.$$
Let $\mathcal{F}(\mathbb{G})\subseteq\mathbb{G}$ be the set of quantum permutations with integer fixed points.
\end{definition}

In the quotient $\pi_{\text{ab}}:C(\mathbb{G})\to C(G)$ to the classical version $G\leq \mathbb{G}$, the number of fixed points becomes  integer-valued:
$$\pi_{\text{ab}}(\operatorname{fix})=\operatorname{fix}_{G}=\sum_{\underset{\lambda\neq N-1}{\lambda=0,1\dots,N}}\lambda \,p_{\lambda},$$
with
$$p_{\lambda}(\sigma)=\begin{cases}
                        1, & \mbox{if $\sigma$ has $\lambda$ fixed points},  \\
                        0, & \mbox{otherwise}.
                      \end{cases}$$
Therefore, random permutations $\varphi_\nu\in S_N^+$ are elements of $\mathcal{F}(S_N^+)$.

\bigskip

There are plenty of concrete examples of genuinely quantum permutations with integer fixed points: e.g. the quantum permutation $E^{11}\circ \pi_{G_0}$ has zero fixed points. So, $\mathcal{F}(S_N^+)$ contains all the elements of $S_N$ in $S_N^+$, and also genuinely quantum permutations.
\begin{proposition}\label{Haarfix}
For $N\geq 4$, the Haar state on $C(S_N^+)$ is not an element of $\mathcal{F}(S_N^+)$. In fact:
$$\omega_h(\mathds{1}_{\{x\}}(\operatorname{fix}))=0\qquad (x\in[0,N]).$$
\end{proposition}

\begin{proof}
  This follows from the fact that for $N\geq 4$ the moments of $\operatorname{fix}$ with respect to the Haar state are the Catalan numbers \cite{bb1}, and thus the corresponding measure is the Marchenko--Pastur law of parameter one, which has no atoms:
  $$\omega_h(\mathds{1}_{\{x\}}(\operatorname{fix}))=\int_{\{x\}} \frac{1}{2\pi}\sqrt{\frac{4}{t}-1}\,dt=0.$$

\end{proof}
\begin{corollary}\label{Haartq}
For $N\geq 4$, the Haar state on $C(S_N^+)$ is truly quantum.
\end{corollary}
\begin{proof}
Assume that $h\in S_N^+$ is mixed:
$$\omega_h(p_C)>0\implies \omega_h(p_{\sigma})>0$$
for some $\sigma\in S_N$. Let $q_\sigma:=\mathds{1}_{S_N^+}-p_\sigma$. Recalling that $p_\sigma$ is central:
$$\omega_h(f)=\omega_h(p_\sigma)\,(\widetilde{p_\sigma}h)(f)+\omega_h(q_\sigma)\,(\widetilde{q_\sigma}h)(f)\qquad(f\in C(S_N^+)^{**}).$$
Note that $\widetilde{p_\sigma}h$ has Birkhoff slice $\Phi(\widetilde{p_\sigma}h)=P_\sigma$, which implies it is a character. By Proposition \ref{charactertwo}, $\widetilde{p_\sigma}h=\operatorname{ev}_\sigma$, which factors through the abelianisation $\pi_{\text{ab}}$:
$$\operatorname{ev}_\sigma(f)=\pi_{\text{ab}}(f)(\sigma)\qquad (f\in C(S_N^+)),$$
 while the extension $\omega_\sigma:C(S_N^+)^{**}\to \mathbb{C}$ factors through $\pi_{\text{ab}}^{**}$. Suppose that $\sigma$ has $\lambda\in\{0,1,\dots,N\}$ fixed points. Using Lemma \ref{fact}, consider, where $p_\lambda=\pi_{\text{ab}}^{**}(\mathds{1}_{\{\lambda\}}(\operatorname{fix}))$,
 \begin{align*}
\omega_\sigma(\mathds{1}_{\{\lambda\}}(\operatorname{fix}))&=p_\lambda(\sigma)=1,
\\\implies \omega_h(\mathds{1}_{\{\lambda\}}(\operatorname{fix}))&=\omega_h(p_\sigma)\,(\widetilde{p_\sigma}h)(\mathds{1}_{\{\lambda\}}(\operatorname{fix}))+\omega_h(q_\sigma)\,(\widetilde{q_\sigma}h)(\mathds{1}_{\{\lambda\}}(\operatorname{fix}))
\\&\geq \omega_h(p_\sigma)\,\omega_\sigma(\mathds{1}_{\{\lambda\}}(\operatorname{fix}))=\omega_h(p_\sigma)>0,
\end{align*}
contradicting Proposition \ref{Haarfix}.
\end{proof}
However, $\mathcal{F}(\mathbb{G})\subseteq \mathbb{G}$ is in general not a Pal set:
\begin{example}
Let $\widehat{S_4}\leq S_5^+$ by:
$$u=\begin{bmatrix}
      u^{(12)} & 0 \\
      0 & u^{(234)}
    \end{bmatrix}.$$
Here $u^{(12)}\in M_2(C(\widehat{S_4}))$ and $u^{(234)}\in M_3(C(\widehat{S_4}))$ are magic unitaries associated with $(12),\,(234)\in S_4$ (Chapter 13, \cite{ba1}). Consider the regular representation:
$$\pi: C(\widehat{S_4})\to B(\mathbb{C}^{24}).$$
With $e\in S_4$ the identity permutation, consider:
$$\pi(\operatorname{fix})=\pi(2e+(12)+(234)+(243)).$$
The spectrum contains $\lambda_{\pm}:=(5\pm\sqrt{17})/2$ (see \cite{mcc}), but consider unit eigenvectors $x_2$ and $x_4\in\mathbb{C}^{24}$ of eigenvalues two and four that give quantum permutations:
 $$\varphi_2=\langle x_2,\pi(\cdot )x_2\rangle\text{ and }\varphi_4=\langle x_4,\pi(\cdot )x_4\rangle,$$
 with two and four fixed points. It can be shown that:
$$\varphi:=\frac12\varphi_2+\frac12\varphi_4$$
is strict, that is $|\varphi(\sigma)|=1$ for $\sigma=e$ only, and therefore as the convolution in $\widehat{S_4}$ is pointwise multiplication,
$$\varphi^{\star k}\to\delta_e,$$
which is the Haar state on $C(\widehat{S_4})$. The Haar state for finite quantum groups such as $\widehat{S_4}$ is faithful, and so where $p_{\lambda_+}$ is the spectral projection associated with the eigenvalue $\lambda_+$:
$$h_{\widehat{S_4}}(p_{\lambda_+})>0,$$
which implies that $(\varphi^{\star k})_{k\geq 0}$ does not converge to an element with integer fixed points, and so  $\mathcal{F}(\widehat{S_4})$ is not a Pal set, and thus neither is $\mathcal{F}(S_N^+)$ for $N\geq 4$.
\end{example}
\begin{example}
In the case of $C(S_N^+)$ ($N\geq 4$), the central algebra $C(S_N^+)_0$ generated by the characters of irreducible unitary representations is commutative \cite{fre}, and generated by $\operatorname{fix}$, and so the central algebra $C(S_N^+)_0\cong C([0,N])$, and the central states are given by Radon probability measures.

\bigskip

The quantum permutation `uniform on quantum transpositions', $\varphi_{\text{tr}}$ from \cite{fre}, is a central state given by:
$$\varphi_{\text{tr}}(f)=f(N-2)\qquad(f\in C(S_N^+)_0)$$
 It has $N-2$ fixed points (see \cite{mcc}) but its convolution powers converge to the Haar state $h\in S_N^+$, which is not in $\mathcal{F}(S_N^+)$ by Proposition \ref{Haarfix}.
\end{example}
\addtocontents{toc}{\protect\setcounter{tocdepth}{0}}
\section*{Appendix: Proof of Vaes' Remark \ref{Vaes}}
\emph{The following result was proven by Stefaan Vaes in a MathOverflow post \cite{Vaes}. Stefaan Vaes kindly permits the proof to be included here.}

\bigskip

The below proof, uses techniques beyond the contents of this paper, the notation for which is summarised very briefly here. The proof applies to all compact quantum groups, and not just those with $C(\mathbb{G})$ a universal algebra of continuous functions. To each compact quantum group $\mathbb{G}$ can be associated a set $\operatorname{Irr}(\mathbb{G})$ of equivalence classes of mutually inequivalent irreducible unitary representations. Each can be viewed as an element $u^\alpha\in C(\mathbb{G})\otimes B(\mathsf{H}_\alpha)$, or, equivalently, as a matrix $u^\alpha\in M_{n_\alpha}(C(\mathbb{G}))$. The tensor product $u^{\alpha\otimes \beta}\in M_{n_\alpha\cdot n_\beta}(C(\mathbb{G}))$ is the matrix with entries $u^\alpha_{ij}u^\beta_{kl}$. Viewing $u^\alpha,\,u^\beta$ as matrices with entries in $C(\mathbb{G})$,
$$T\in \operatorname{Mor}_{\mathbb{G}}(\alpha,\beta)\iff Tu^\alpha=u^\beta T.$$  Denote by $c_0(\widehat{\mathbb{G}})$ the $c_0$ direct sum of the matrix algebras $B(\mathsf{H}_\alpha)$, $\alpha \in \operatorname{Irr}(\mathbb{G})$. Denote by $\ell^\infty(\widehat{\mathbb{G}})$ the $\ell^\infty$ direct sum. Denote by $W \in M(C(\mathbb{G}) \otimes c_0(\widehat{\mathbb{G}}))$ the corresponding direct sum of unitary representations of $\mathbb{G}$. Similarly define $V \in M(C_{\text{u}}(\mathbb{H}) \otimes c_0(\widehat{\mathbb{H}}))$ for the compact quantum group $\mathbb{H}$.

\begin{proposition}
  Let $\mathbb{G}$ and $\mathbb{H}$ be compact quantum groups and $\pi : C(\mathbb{G}) \to C_{\text{u}}(\mathbb{H})$ a surjective unital $*$-homomorphism satisfying $(\pi \otimes \pi) \circ \Delta = \Delta_{\mathbb{H}} \circ \pi$.  Denote by $h_{\mathbb{H}} = h_{C_{\text{u}}(\mathbb{H})} \circ \pi$. Let $\varphi$ be a state on $C(\mathbb{G})$. Then the following statements are equivalent.
\begin{enumerate}
\item[(i)] There exists a state $\psi$ on $C_{\text{u}}(\mathbb{H})$ such that $\varphi = \psi \circ \pi$.
\item[(ii)] $\varphi$ is in the quasi-subgroup generated by $h_{\mathbb{H}}$,
\item[(iii)] $\varphi \star h_{\mathbb{H}}=h_{\mathbb{H}}$.
\end{enumerate}
\end{proposition}
\begin{proof}
The implications $(i) \Rightarrow (ii) \Rightarrow (iii)$ are trivial. Assume that (iii) holds. Denote by $c_0(\widehat{\mathbb{G}})$ the $c_0$ direct sum of the matrix algebras $B(H_\alpha)$, $\alpha \in \operatorname{Irr}(\mathbb{G})$. Denote by $\ell^\infty(\widehat{\mathbb{G}})$ the $\ell^\infty$ direct sum. Denote by $W \in M(C(\mathbb{G}) \otimes c_0(\widehat{\mathbb{G}}))$ the corresponding direct sum of unitary representations of $\mathbb{G}$. Similarly define $V \in M(C(\mathbb{H}) \otimes c_0(\widehat{\mathbb{H}}))$ for the compact quantum group $\mathbb{H}$. The morphism $\pi$ dualizes to a morphism $\widehat{\pi} : \ell^\infty(\widehat{\mathbb{H}}) \to \ell^\infty(\widehat{\mathbb{G}})$ satisfying $(\pi \otimes \mathrm{id})(W) = (\mathrm{id} \otimes \widehat{\pi})(V)$.

\bigskip

If given, for $i = 1,2$, unitary representations $U_i \in M(C(\mathbb{G}) \otimes \mathcal{K}(H_i))$ of $\mathbb{G}$ on Hilbert spaces $H_i$, define the intertwiner space $\operatorname{Mor}_\mathbb{G}(U_1,U_2)$ as the space of bounded operators $A \in B(H_1,H_2)$ satisfying $(1 \otimes A) U_1 = U_2 (1 \otimes A)$. Every unitary representation $U \in M(C(\mathbb{G}) \otimes \mathcal{K}(H))$ of $\mathbb{G}$ can be restricted to a unitary representation $U_\mathbb{H} := (\pi \otimes \mathrm{id})(U)$ of $\mathbb{H}$. Denote by $\operatorname{Mor}_\mathbb{H}(U_1,U_2)$ the space of intertwiners between these restrictions, so that $A \in \operatorname{Mor}_\mathbb{H}(U_1,U_2)$ if and only if $A \in B(H_1,H_2)$ and $(1 \otimes A) U_{1,\mathbb{H}} = U_{2,\mathbb{H}} (1 \otimes A)$. Note that $\operatorname{Mor}_\mathbb{G}(U_1,U_2) \subset \operatorname{Mor}_\mathbb{H}(U_1,U_2)$. Write $\operatorname{End}_\mathbb{H}(U) = \operatorname{Mor}_\mathbb{H}(U,U)$.

\bigskip

Note that the formula $U = (\mathrm{id} \otimes \theta)(V)$ defines a bijective correspondence between unitary representations $U \in M(C(\mathbb{H}) \otimes \mathcal{K}(H))$ of $\mathbb{H}$ and unital normal $*$-homomorphisms $\theta : \ell^\infty(\widehat{\mathbb{H}}) \to B(H)$. Then $\operatorname{End}_\mathbb{H}(U) = \theta(\ell^\infty(\widehat{\mathbb{H}}))'$ and $\theta(\ell^\infty(\widehat{\mathbb{H}})) = \operatorname{End}_\mathbb{H}(U)'$. It follows that for $T = (T_\alpha)_{\alpha \in \operatorname{Irr}(\mathbb{G})}$ in $\ell^\infty(\widehat{\mathbb{G}})$,
\begin{equation}\label{eq.rel-com}
T \in \widehat{\pi}(\ell^\infty(\widehat{\mathbb{H}})) \quad\text{iff}\quad T_\alpha Z = Z T_\beta \;\;\text{for all $\alpha,\beta \in \operatorname{Irr}(\mathbb{G})$ and $Z \in \operatorname{Mor}_\mathbb{H}(\beta,\alpha)$.}
\end{equation}

Denote by $p_\varepsilon \in c_0(\widehat{\mathbb{H}})$ the minimal central projection that corresponds to the trivial representation of $\mathbb{H}$. Write $q = \widehat{\pi}(p_\varepsilon)$. Note that $(h_1 \otimes \mathrm{id})(W) = q$. Define $T \in \ell^\infty(\widehat{\mathbb{G}})$ by $T = (\varphi \otimes \mathrm{id})(W)$. Since $(\Delta \otimes \mathrm{id})(W) = W_{13} W_{23}$ and $\varphi \ast h_1 = h_1$, it follows that $q = T q$. Write $R = W(1 \otimes q) - (1 \otimes q)$. Because
$$R^* R = 2 (1 \otimes q) - (1 \otimes q) W (1 \otimes q) - (1 \otimes q) W^* (1 \otimes q) \; ,$$
it is the case that $(\varphi \otimes \mathrm{id})(R^* R) = 0$. Since $\varphi$ is a state on a C$^*$-algebra, it follows that
$$(\varphi \otimes \mathrm{id} \otimes \mathrm{id})(W_{12} R_{13}) = 0 \; .$$
This means that
\begin{equation}\label{eq.formula}
(\varphi \otimes \mathrm{id} \otimes \mathrm{id})(W_{12} W_{13}) \, (1 \otimes q) = T \otimes q \; .
\end{equation}
Write $T$ as the direct sum of the matrices $T_\alpha \in B(H_\alpha)$ for $\alpha \in \operatorname{Irr}(\mathbb{G})$. Similarly denote the components of $q$ by $q_\gamma$ for every $\gamma \in \operatorname{Irr}(\mathbb{G})$. Take arbitrary $\alpha,\beta,\gamma \in \operatorname{Irr}(\mathbb{G})$. Consider the component in $B(H_\beta) \otimes B(H_\gamma)$ of \eqref{eq.formula}. Since $W$ is the direct sum of the unitary representations $U_\zeta \in C(\mathbb{G}) \otimes B(H_\zeta)$, $\zeta \in \operatorname{Irr}(\mathbb{G})$,
\begin{equation}\label{eq.new-formula}
(\varphi \otimes \mathrm{id} \otimes \mathrm{id})(U_{\beta,12} U_{\gamma,13}) \, (1 \otimes q_\gamma) = T_\beta \otimes q_\gamma \; .
\end{equation}
Take arbitrary $\alpha \in \operatorname{Irr}(\mathbb{G})$, $X \in \operatorname{Mor}_\mathbb{G}(\alpha,\beta \otimes \gamma)$ and $Y \in \operatorname{Mor}_\mathbb{H}(\varepsilon,\gamma)$. Multiplying \eqref{eq.new-formula} on the left by $X^*$ and on the right by $1 \otimes Y$,
\begin{equation}\label{eq.new-new-formula}
(\varphi \otimes \mathrm{id} \otimes \mathrm{id})((1 \otimes X^*) U_{\beta,12} U_{\gamma,13}) \, (1 \otimes q_\gamma Y) = X^*(T_\beta \otimes q_\gamma Y) \; .
\end{equation}
Note that $q_\gamma Y = \widehat{\pi}(p_\varepsilon)_\gamma Y = Y \widehat{\pi}(p_\varepsilon)_\varepsilon = Y$. Therefore, the left hand side of \eqref{eq.new-new-formula} equals
\begin{align*}
(\varphi \otimes \mathrm{id} \otimes \mathrm{id})((1 \otimes X^*) U_{\beta,12} U_{\gamma,13}) \, (1 \otimes Y) &= (\varphi \otimes \mathrm{id})(U_\alpha (1 \otimes X^*)) \, (1 \otimes Y) \\ &= (\varphi \otimes \mathrm{id})(U_\alpha) \, X^* (1 \otimes Y) = T_\alpha \, X^* (1 \otimes Y) \; ,
\end{align*}
while the right hand side of \eqref{eq.new-new-formula} equals
$$X^* (T_\beta \otimes Y) = X^* (1 \otimes Y) \, T_\beta \; .$$
Thus
$$T_\alpha \, X^*(1 \otimes Y) = X^*(1 \otimes Y) \, T_\beta$$
for all $\alpha,\beta,\gamma \in \operatorname{Irr}(\mathbb{G})$ and for all $X \in \operatorname{Mor}_\mathbb{G}(\alpha,\beta \otimes \gamma)$ and $Y \in \operatorname{Mor}_\mathbb{H}(\varepsilon,\gamma)$. By taking linear combinations, the same holds if $\gamma$ is any finite dimensional unitary representation of $\mathbb{G}$, not necessarily irreducible.

\bigskip

Now $T_\alpha \, Z = Z \, T_\beta$ for all $\alpha,\beta \in \operatorname{Irr}(\mathbb{G})$ and $Z \in \operatorname{Mor}_\mathbb{H}(\beta,\alpha)$. To prove this, choose solutions of the conjugate equations for $\alpha,\beta \in \operatorname{Irr}(\mathbb{G})$, given by $t \in \operatorname{Mor}_\mathbb{G}(\varepsilon,\overline{\beta} \otimes \beta)$ and $s \in \operatorname{Mor}_\mathbb{G}(\varepsilon,\beta \otimes \overline{\beta})$. Writing $\gamma = \overline{\beta} \otimes \alpha$, $X = s \otimes 1$ and $Y = (1 \otimes Z)t$, it follows that $X^*(1 \otimes Y) = Z$ and the claim follows. By this claim and by \eqref{eq.rel-com}, this\ means that $T = \widehat{\pi}(S)$ for some $S \in \ell^\infty(\widehat{\mathbb{H}})$.

\bigskip

Every $a \in \mathcal{O}(\mathbb{G})$ is of the form $a = (\mathrm{id} \otimes \omega)(W)$ where $\omega$ is a uniquely determined, ``finitely supported'' functional on $\ell^\infty(\mathbb{G})$. Note that $\pi(a) = 0$ if and only if $\omega \circ \widehat{\pi} = 0$. Since $(\varphi \otimes \mathrm{id})(W) = T = \widehat{\pi}(S)$,  there is a well defined linear functional $\psi : \mathcal{O}(\mathbb{H}) \to \mathbb{C}$ such that $\psi(\pi(a)) = \varphi(a)$ for all $a \in \mathcal{O}(\mathbb{G})$.

\bigskip

Then $\psi(1) = \varphi(1) = 1$ and, because $\pi : \mathcal{O}(\mathbb{G}) \to \mathcal{O}(\mathbb{H})$ is surjective, $\psi(b^* b) \geq 0$ for all $b \in \mathcal{O}(\mathbb{H})$. Since $C(\mathbb{H})$ is the universal C$^*$-algebra of the compact quantum group $\mathbb{H}$, it follows that $\psi$ uniquely extends to a state on $C(\mathbb{H})$ (still denoted $\psi$). By construction, $\varphi = \psi \circ \pi$.
\end{proof}
\subsection*{Acknowledgement}
Some of this work goes back to discussions with Teo Banica. Thanks also to Matthew Daws for helping with Section \ref{bid}, and Ruy Exel with the argument in Theorem \ref{wfct} (ii). Special thanks to Stefaan Vaes for going above and beyond in assisting with Remark \ref{Vaes}, and for kindly allowing the teased out proof to be reproduced here. Finally thank you to the referees for their many helpful suggestions.

\end{document}